\documentclass{elsarticle}
\newcommand{\prodint}[1]{\left\langle {#1}\right\rangle}
\RequirePackage{fix-cm}
\usepackage{graphicx}
\usepackage{amsmath, amsopn,amstext,amscd,amsfonts,amssymb}
\usepackage{dsfont}
\usepackage{comment}
\usepackage[active]{srcltx}
\usepackage{graphicx, epsfig, subfig}
\usepackage{infix-RPN}
\usepackage{bm}
\usepackage{siunitx}
\usepackage{tikz}

\usepackage{textcomp}

\usepackage{algorithm}

\makeatletter
\def\BState{\State\hskip-\ALG@thistlm}
\makeatother

\def\downbar#1{
\setbox10=\hbox{$#1$}
            \dimen10=\ht10 \advance\dimen10 by 2.5pt
            \ifdim \dimen10<15pt 
               \advance\dimen10 by -0.5pt
               \dimen11=\dimen10
               \advance\dimen10 by 2.5pt
               \lower \dimen11
            \else \lower \ht10 \fi
            \hbox {\hskip 1.5pt \vrule height \dimen10 depth \dp10}}
\def\upbar#1{
\setbox10=\hbox{$#1$}
            \dimen10=\ht10 \advance\dimen10 by \dp10 \advance\dimen10 by 2.5pt
            \ifdim \dimen10<15pt 
                \advance\dimen10 by 2pt \fi
            \raise 2.5pt \hbox {\hskip -1.5pt \vrule height \dimen10}}

\usepackage{multicol}
\usepackage{colortbl}
\usepackage{exscale}
\usepackage{epsfig}
\usepackage{pst-plot,pst-infixplot,pstricks,graphicx}
\usepackage{amssymb,latexsym,amsthm,amsfonts,color,fancyhdr,mathrsfs}
\usepackage{lscape}
\usepackage{rotating}
\usepackage{caption}

\newrgbcolor{verde}{0.0 0.5 0.0}
\newrgbcolor{morado}{.5 0 .5}
\newrgbcolor{azul}{0 0 .5}
\newrgbcolor{rojo}{0.5 0 0}
\newrgbcolor{brown}{0.55 0.14 0.14}
\newrgbcolor{orange}{1.0 0.3 0.0}
\newrgbcolor{seagreen}{0.33 1.00 0.62}
\newrgbcolor{darkgreen}{0.00 0.39 0.00}

\newtheorem{definition}{\bf Definition}[section]
\newtheorem{theorem}{\bf Theorem}[section]
\newtheorem{proposition}{\bf Proposition}[section]
\newtheorem{lemma}{\bf Lemma }[section]

\newtheorem{corollary}{\bf Corollary}[section]
\newtheorem{remark}{\bf Remark}[section]

\numberwithin{equation}{section}

\journal{}

\begin{document}

\begin{frontmatter}


\title{Classical orthogonal polynomials revisited\footnote{My friend and coauthor Professor J. Petronilho passed away  on August 27, 2021. In the academic years 2016/17 and 2020/21 we taught a PhD course on Orthogonal Polynomials and Special Functions at the Universities of Coimbra and Porto (Portugal). This manuscript go back to the beginning of 2015 and it was later including in a monograph \cite{P} on the subject that Professor Petronilho finished  in 2016. (His monographs was during all these years actively used by students and colleagues, mainly from Portugal and Spain.) Probably, for this reason or the fact that we had started working on \cite{CMP}, which allow us to rewrite many results in a more general framework, we  shelved this project. And this would have been its final destination if I had never read some recent publications where the same topic is addressed. But after that, I felt that is time to dust off this manuscript, allowing active readers the freedom to draw their own conclusions.  Finally, I thank Professor Petronilho's youngest daughter, In\^es, who found these old files and sent them to me.
}}

%

%


\author{K. Castillo} \ead{kenier@mat.uc.pt}
\author{J. Petronilho}

\address{CMUC, University of Coimbra, Department of Mathematics.\\ EC Santa Cruz, 3001-501 Coimbra, Portugal.}

\date{\today}

\begin{abstract}

This manuscript contains a small portion of the algebraic theory of orthogonal polynomials developed by Maroni and their applicability to the study and characterization of the classical families, namely Hermite, Laguerre, Jacobi, and Bessel polynomials. It is presented a cyclical proof of some of the most relevant characterizations, particularly those due to Al-Salam and Chihara, Bochner, Hahn, Maroni, and McCarthy. Two apparently new characterizations are also added. Moreover, it is proved through an equivalence relation that, up to constant factors and affine changes of variables, the four families of polynomials named above are the only families of classical orthogonal polynomials.
\end{abstract}

\begin{keyword}
Moment linear functionals \sep classical orthogonal polynomials
\sep algebraic theory of orthogonal polynomials
\MSC[2010] 42C05 \sep 33C45
\end{keyword}

\end{frontmatter}

%

\section{Orthogonal polynomials: The algebraic theory}

Denote by $\mathrm{M}_m(\mathrm{A})$ the ring of all square matrices of order $m$ over the commutative ring with unit $\mathrm{A}$.  Consider the free modulo $\mathcal{M}=\mathrm{M}_m(\mathrm{A})[X]$ of all ``polynomials'' in one indeterminate $X$ with coefficients in $\mathrm{M}_m(\mathrm{A})$. (Assume that  the indeterminate is contained in the centre of $\mathcal{M}$.)  Any free system of polynomials over $\mathrm{M}_m(\mathrm{A})$ is basis of $\mathcal{M}$. Moreover,  the dual of a finite generated free module is a finitely generated free module. However, the dual system associated with a free sequence $(p_n)_{n\geq 0}$ in $\mathcal{M}$, $(\mathbf{a}_n)_{n \geq 0}$, is not  a system of generators of its algebraic dual, $\mathcal{M}^*$, in general $($cf. \cite[VII, \S 3, Exercise 10]{AI}$)$. But if we consider, for instance, $\mathcal{P}=\mathbb{C}[X]$, instead of $\mathcal{M}$, endowed with an ``appropriate" topology, then  $\mathcal{P}^*=\mathcal{P}'$ $($cf. \cite[Exercise 13.1, p. 134]{Treves1967}$)$, $\mathcal{P}'$ being the topological dual $\mathcal{P}$. Hence 
\begin{align}\label{MM}
\mathbf{u}=\sum_{n=0}^\infty \prodint{p_n, \mathbf{u}} \mathbf{a}_n,
\end{align}
for all $\mathbf{u} \in \mathcal{P}'$. (Recall that for every pair of elements $x \in \mathrm{E}$ and $\mathbf{x} \in \mathrm{E}^*$, the element $\mathbf{x}(x)$ of $\mathrm{E}$ is denote by $\prodint{x, \mathbf{x}}$, where $\mathrm{E}$ is a left $\mathrm{A}$-module whose domain of operators is $\mathrm{A}$.) This observation is the cornerstone of Rota's umbral calculus $($cf. \cite{Roman}$)$ and the algebraic theory of orthogonal polynomials (OP) founded by Maroni $($cf. \cite{Maroni1985, Maroni1988, Maroni1991a}$)$. For further reading on how to re-establish the ``symmetry" between an infinite dimensional vector space and its dual see \cite[II, \S 6]{TVS}. 

This manuscript contains a short exposition of Maroni's approach on OP and their applicability to the study and characterization of the classical families of OP. A must read on this topic are  \cite{Maroni1991b, Maroni1993, Maroni1994}. Even a proof of  Theorem \ref{ThmClassicalOP} can be found,  in one way or another, in the works of Maroni. However, among other questions of pedagogical nature, in this survey we present a cyclical proof of our main results, Theorem \ref{ThmClassicalOP}. Moreover, the characterizations (C4) and (C4') are apparently new ones. While is true that  any classical functional is equivalent to one of the canonical forms given in Table \ref{Table1} below, Theorem \ref{canonic-forms-classical} rigorously reflects this property and, as far as we know, it is not available in the literatura. Many results are stated without proof, either because they are simple to prove or because they can be easily found in the literature. Other results that fit in these two categories are, however, proved when the proof methods are different or more attractive, from our point of view, to the existing one. The positive definite case, widely discussed in the literature, is intentionally omitted. Finally, as the reader will notice, some results are ``purely" algebraic. In any case, we will not make distinction between $\mathcal{P}^*$ and $\mathcal{P}'$ after Proposition \ref{P*=Plinha} below. 

\subsection{The spaces $\mathcal{P}$ and $\mathcal{P}'$}

It is useful to consider OP as test functions living in an appropriate locally convex space (LCS), which we denote by $\mathcal{P}$. This LCS is the set of all polynomials (with real or complex coefficients) endowed with a strict inductive limit topology, so that
\begin{equation}\label{P=indlimPn}
\mathcal{P}=\bigcup_{n=0}^\infty\mathcal{P}_n=\mbox{\rm ind\,lim}_n\,\mathcal{P}_n\; ,
\end{equation}
where $\mathcal{P}_n$ is the space of all polynomials of degree at most $n$.
(For the sake of simplicity, we do not distinguish between {\it polynomial} and {\it polynomial function}.)
Since $\mathcal{P}_n$ is a finite dimensional vector space, all its norms are equivalent,
so there is no need to specify any particular one.
For the development of the theory to be presented here it is not important to know much about the topology (the definition and basic properties of LCS, including inductive limit topologies, can be found, e.g., in Chapter V of
Reed and Simon's book \cite{ReedSimon1972}, but the reader should keep in mind that the reason why such topology is introduced is because it implies the following fundamental property: 

\begin{proposition}\label{P*=Plinha}
Let $\mathcal{P}=\mbox{\rm ind\,lim}_n\,\mathcal{P}_n$, as in $(\ref{P=indlimPn})$,
and let $\mathcal{P}^*$ and $\mathcal{P}'$ be the algebraic and the
topological duals of $\mathcal{P}$, respectively. Then
\begin{equation}\label{P*equalsPprime}
\mathcal{P}'=\mathcal{P}^*\,.
\end{equation}
\end{proposition}

\begin{proof}
Obviously, $\mathcal{P}'\subseteq\mathcal{P}^*$.
To prove that $\mathcal{P}^*\subseteq\mathcal{P}'$, take ${\bf u}\in\mathcal{P}^*$.
From the basic properties of the inductive limit topologies, to prove that ${\bf u}\in\mathcal{P}'$
it suffices to show that the restriction ${\bf u}|\mathcal{P}_n$ is continuous for every $n$.
But this is a trivial assertion, since ${\bf u}|\mathcal{P}_n$ is a linear functional
defined on a finite dimensional normed space.
\end{proof}

Equality (\ref{P*equalsPprime}) means that every linear functional defined in $\mathcal{P}$ is continuous (for the strict inductive limit topology in $\mathcal{P}$). This is a ``curious" property, because, for instance, \eqref{P*equalsPprime} is not true for a normed vector space $N$. Indeed, $N'=N^*$ if ${\rm dim}\,N<\infty$, whilst $N'\neq N^*$ whenever ${\rm dim}\,N=\infty$. Note that being a strict inductive limit of the spaces $\mathcal{P}_n$, and taking into account that each $\mathcal{P}_n$ is a proper closed subspace of $\mathcal{P}_{n+1}$ (so that $\mathcal{P}$ is indeed an hyper strict inductive limit of the spaces $\mathcal{P}_n$), the general theory of LCS ensures that $\mathcal{P}$ cannot be a metrizable space, and so \emph{a fortiori} it is not a normed space---or, to be more precise, it is not possible to provide $\mathcal{P}$ with a norm that generates in it the above inductive limit topology.

In $\mathcal{P}'$ we consider the weak dual topology, which, by definition, is generated by the family of semi-norms
$s_p:\mathcal{P}'\to[0,+\infty[$, $p\in\mathcal{P}$, defined by
\begin{equation}\label{semip1}
s_p\big({\bf u}\big)=|\langle {\bf u},p\rangle|\; ,\quad {\bf u}\in\mathcal{P}'\; .
\end{equation}
It turns out that this family of semi-norms $s_p$ is equivalent to the family of semi-norms
$|\cdot|_n:\mathcal{P}'\to[0,+\infty[$, $n\in\mathbb{N}_0$, defined by
\begin{equation}\label{sharp1}
|{\bf u}|_n=\max_{0\leq k\leq n}|\langle {\bf u},x^k\rangle|\; ,\quad {\bf u}\in\mathcal{P}'\; .
\end{equation}
Indeed, the following proposition holds.
\begin{proposition}\label{Snumeravel}
$\mathcal{S}=\{s_p:p\in\mathcal{P}\}$ and $\mathcal{S}_\sharp=\{|\cdot|_n:n\in\mathbb{N}_0\}$,
with $s_p$ and $|\cdot|_n$ given by $(\ref{semip1})$--$(\ref{sharp1})$,
are equivalent families of seminorms in  $\mathcal{P}'$, provided $\mathcal{P}=\mbox{\rm ind\,lim}_n\,\mathcal{P}_n$.
\end{proposition}

\begin{proof}
Given $p\in\mathcal{P}$, putting $p(x)=\sum_{j=0}^na_jx^j$ and $C(p)=\sum_{j=0}^n|a_j|$, we have
$$
s_{p}({\bf u})=|\langle {\bf u},p\rangle|=\Big| \sum_{j=0}^na_j\langle {\bf u},x^j\rangle \Big|
\leq C(p)|{\bf u}|_n\,,\quad \forall{\bf u}\in\mathcal{P}'\,.
$$
On the other hand, given $n\in\mathbb{N}_0$, setting $p_j(x)=x^j$ ($j=0,1,\ldots,n$), we have
$$
|{\bf u}|_n=\max_{0\leq j\leq n}|\langle {\bf u},x^j\rangle|
\leq\sum_{j=0}^n|\langle {\bf u},x^j\rangle|=\sum_{j=0}^n s_{p_j}({\bf u})
\,,\quad \forall{\bf u}\in\mathcal{P}'\,.
$$
Thus, $\mathcal{S}$ and $\mathcal{S}_\sharp$ are equivalent families of semi-norms
(see \cite[p.126]{ReedSimon1972}).
\end{proof}

\begin{remark}
Since $\mathcal{S}_\sharp$ is a countable family of seminorms, then
$\mathcal{P}'$ is a metrizable space, a metric being given by
\begin{equation*}
\varrho({\bf u},{\bf v})=\sum_{n=0}^\infty\frac{1}{2^n}\frac{|{\bf u}-{\bf v}|_n}{1+|{\bf u}-{\bf v}|_n}\; ,
\quad {\bf u},{\bf v}\in\mathcal{P}'\; .
\end{equation*}
Moreover, $\mathcal{P}'$ is a Fr\'echet space.
\end{remark}

\subsection{Dual basis in $\mathcal{P}^*$}
In $\mathcal{P}^*$, addition and multiplications by scalars can be defined by
\begin{align*}
\prodint{\mathbf{u}+\mathbf{v}, x^n}&=\prodint{\mathbf{u}, x^n}+\prodint{\mathbf{v}, x^n}, \quad \mathbf{v} \in \mathcal{P}^*,\\
\prodint{\alpha \mathbf{u}, x^n}&=\alpha \prodint{\mathbf{u}, x^n}
\end{align*}
for all $n\in\mathbb{N}_0$. $\mathcal{P}^*$, endowed with these operations, is a vector space over $\mathbb{C}$.  In the vector space $\mathcal{P}^*$, the identity for the additivity is denoted by $\mathbf{0}$ and called the zero (or the null element). The zero is therefore defined by the relation $\prodint{\mathbf{0}, x^n}=0$ for all $n \in \mathbb{N}_0$. Of course, the elements of $\mathcal{P}^*$ can not only be added, but also multiplied (the {\em Cauchy product}) in order to make the vector space $\mathcal{P}^*$ into an {\em algebra}.  Since we work mainly on $\mathcal{P}^*$ instead of $\mathcal{P}$, it would be explicitly build bases in $\mathcal{P}^*$. This makes sense, since (\ref{P*equalsPprime}) allow us writing expansions (finite or infinite sums) of the elements of $\mathcal{P}^*$ in terms of the elements of a given basis, in the sense of the weak dual topology. Such basis in $\mathcal{P}^*$ may be achieved in a natural way, by analogy with the case of finitely generated free modules. A {\sl simple set} in $\mathcal{P}$ is a sequence of polynomials, $\{R_n\}_{n\geq0}$, such that ${\rm deg}\,R_n=n$ for every $n\in\mathbb{N}_0$ (where $R_0\equiv\mbox{\rm const.}\neq0$). To any simple set in $\mathcal{P}$, $\{R_n\}_{n\geq0}$, we may associate a {\sl dual basis}, which, by definition, is a sequence of linear functionals $\{{\bf a}_n\}_{n\geq0}$, being ${\bf a}_n:\mathcal{P}\to\mathbb{C}$, such that
$$
\langle {\bf a}_n,R_k\rangle=\delta_{n,k},\quad n,k=0,1,2,\dots\; ,
$$
where $\delta_{n,k}$ represents the Kronecker symbol ($\delta_{n,k}=1$ if $n=k$; $\delta_{n,k}=0$ if $n\neq k$).

\begin{proposition}\label{expDB1}
Let $\{R_n\}_{n\geq0}$ be a simple set in $\mathcal{P}$ and $\{{\bf a}_n\}_{n\geq0}$ the associated dual basis.
Let ${\bf u}\in\mathcal{P}^*$. Then
\begin{equation}\label{expDB2}
{\bf u}=\sum_{n=0}^\infty\langle{\bf u},R_n\rangle\,{\bf a}_n\, , 
\end{equation}
in the sense of the weak dual topology in $\mathcal{P}'$.
\end{proposition}

\begin{proof}
Notice first that the assertion makes sense, according with (\ref{P*equalsPprime}).
To prove it, fix $N\in\mathbb{N}$ and let
$${\bf s}_N=\sum_{n=0}^{N-1}\lambda_n{\bf a}_n,\quad \lambda_n=\langle{\bf u},R_n\rangle$$
be the partial sum of order $N$ of the series appearing in (\ref{expDB2}).
We need to show that
$$
\lim_{N\to\infty}\langle {\bf s}_N-{\bf u},p\rangle=0\;,\quad\forall p\in\mathcal{P}\;.
$$
Clearly, it suffices to prove that this equality holds for $p\in\{R_0,R_1,R_2,\dots\}$.
Indeed, fix $k\in\mathbb{N}_0$. Then, for $N>k$,
$$
\langle {\bf s}_N-{\bf u},R_k\rangle
=\sum_{n=0}^{N-1}\langle{\bf u},R_n\rangle\langle{\bf a}_n,R_k\rangle -\langle{\bf u},R_k\rangle
=0\; ,
$$
hence $\;\lim_{N\to\infty}\langle {\bf s}_N-{\bf u},R_k\rangle=0$.
\end{proof}

\subsection{Basic operations in $\mathcal{P}$ and $\mathcal{P}'$}
Given a functional ${\bf u}\in\mathcal{P}'$, we will denote by
$$
u_n=\langle{\bf u},x^n\rangle\;,\quad n\in\mathbb{N}_0\; ,
$$
the {\it moment} of order $n$ of ${\bf u}$.
Clearly, if ${\bf u}$ and ${\bf v}$ are two functionals in $\mathcal{P}'$
such that the corresponding sequences of moments satisfy
$u_n=v_n$ for all $n\in\mathbb{N}_0$, then ${\bf u}={\bf v}$.
Therefore, each functional ${\bf u}\in\mathcal{P}'$ is uniquely determined
by its sequence of moments. Define operators $M_\phi$ and $T$, from $\mathcal{P}$ into $\mathcal{P}$, by
\begin{equation}\label{thetac}
M_\phi p(x)= \phi(x)p(x)\;,\quad Tp(x)= -p'(x)\;,
\end{equation}
where $\phi\in\mathcal{P}$ (fixed) and $'$ denotes derivative with respect to $x$.
Let $M_\phi'$ and $T'$ be the corresponding dual operators.
For each ${\bf u}\in\mathcal{P}'$, the images $M_\phi'{\bf u}$ and $T'{\bf u}$
are elements (functionals) in $\mathcal{P}'$,
hereafter denoted by $\phi\,{\bf u}$ and $D{\bf u}$.

\begin{definition}\label{def-left-mult}
Let ${\bf u}\in\mathcal{P}'$, $\phi\in\mathcal{P}$, and $c\in\mathbb{C}$.
\begin{enumerate}
\item[{\rm (i)}]
the {\sl left multiplication} of ${\bf u}$ by $\phi$, denoted by $\phi{\bf u}$,
is the functional in $\mathcal{P}'$ defined by
$$
\langle \phi {\bf u}, p\rangle=\langle {\bf u}, \phi p\rangle\;,\quad p\in\mathcal{P}\;;
$$
\item[{\rm (ii)}]
the {\sl derivative} of ${\bf u}$, denoted by $D{\bf u}$,
is the functional in $\mathcal{P}'$ defined by
$$
\langle D{\bf u}, p\rangle=-\langle{\bf u}, p'\rangle\;,\quad p\in\mathcal{P}\;;
$$
\end{enumerate}
\end{definition}

Note that these definitions, introduced by duality with respect to the
operators defined in (\ref{thetac}), are in accordance with those usually given in the
Theory of Distributions (this explains the minus sign appearing in the second definition).
Note also that
$$
D(\phi{\bf u})=\phi'{\bf u}+\phi\,D{\bf u}\,,\quad
{\bf u}\in\mathcal{P}'\;,\;\phi\in\mathcal{P}\;.
$$

\begin{definition}[translation operators]\label{taub}
Let $b\in\mathbb{C}$.
\begin{enumerate}
\item[{\rm (i)}]
The {\sl translator operator} on $\mathcal{P}$ is $\tau_b:\mathcal{P}\to\mathcal{P}$ $(p\mapsto\tau_bp)$ defined by
\begin{equation}\label{def-taub}
\tau_bp(x)=p(x-b)\;,\quad p\in\mathcal{P}\;;
\end{equation}
\item[{\rm (ii)}]
The {\sl translator operator} on $\mathcal{P}'$ is $\bm{\tau}_b=\tau_{-b}^{\,\prime}$, i.e., $\bm{\tau}_b:\mathcal{P}'\to\mathcal{P}'$
is the dual operator of $\tau_{-b}$, so that
\begin{equation}\label{def-taub-dual}
\langle\bm{\tau}_b{\bf u},p\rangle=\langle{\bf u},\tau_{-b}p\rangle=\langle{\bf u},p(x+b)\rangle\;,
\quad {\bf u}\in\mathcal{P}'\;,\quad p\in\mathcal{P}\;.
\end{equation}
\end{enumerate}
\end{definition}

\begin{definition}[homothetic operators]\label{ha}
Let $a\in\mathbb{C}\setminus\{0\}$.
\begin{enumerate}
\item[{\rm (i)}]
The {\sl homothetic operator} on $\mathcal{P}$ is $h_a:\mathcal{P}\to\mathcal{P}$ $(p\mapsto h_ap)$ defined by
\begin{equation}\label{def-ha}
h_ap(x)=p(ax)\;,\quad p\in\mathcal{P}\;.
\end{equation}
\item[{\rm (ii)}]
The {\sl homothetic operator} on $\mathcal{P}'$ is $\bm{h}_a=h_a^{\,\prime}$,
i.e., $\bm{h}_a:\mathcal{P}'\to\mathcal{P}'$ is the dual operator of $h_a$, so that
\begin{equation}\label{def-ha-dual}
\langle\bm{h}_a{\bf u},p\rangle=\langle{\bf u},h_ap\rangle=\langle{\bf u},p(ax)\rangle\;,
\quad {\bf u}\in\mathcal{P}'\;,\quad p\in\mathcal{P}\;.
\end{equation}
\end{enumerate}
\end{definition}


%
%

\begin{proposition}\label{uv-equiv}
Let $\{P_n\}_{n\geq0}$ be a simple set in $\mathcal{P}$ and $\{{\bf a}_n\}_{n\geq0}$ its associated dual basis.
Let $a\in\mathbb{C}\setminus\{0\}$ and $b\in\mathbb{C}$. Define
\begin{equation}\label{QnAfimPn}
Q_n=a^{-n}\big(h_a\circ\tau_{-b}\big)P_n\;,\quad n=0,1,2,\ldots.
\end{equation}
Then $\{Q_n\}_{n\geq0}$ is a simple set in $\mathcal{P}$, and
its dual basis, $\{{\bf b}_n\}_{n\geq0}$, is given by
\begin{equation}\label{dual-bas-Qn}
{\bf b}_n=a^n\big(\bm{h}_{a^{-1}}\circ\bm{\tau}_{-b}\big){\bf a}_n\;,\quad n=0,1,2,\ldots\;.
\end{equation}
\end{proposition}

\begin{remark}\em
The polynomial $Q_n$ in (\ref{QnAfimPn}) is indeed
\begin{equation}
Q_n(x)=a^{-n}P_n(ax+b)\;,\quad n=0,1,2,\ldots\;,
\end{equation}
so that $Q_n$ is obtained from $P_n$ by an affine change of the variable,
being $Q_n$ normalized so that it becomes a monic polynomial whenever $P_n$ is monic.
\end{remark}

\begin{proposition}\label{prop-uvzero}
Let ${\bf u}\in\mathcal{P}'$ and $p,q\in\mathcal{P}\setminus\{0\}$, and denote by $Z_p$ and $Z_q$ the zeros of $p$ and $q$, respectively.
Then the following property holds:
\begin{equation}\label{pu0qu}
Z_p\cap Z_q=\emptyset\quad\wedge\quad p{\bf u}=q{\bf u}={\bf 0}\qquad\Rightarrow\qquad{\bf u}={\bf 0}\;.
\end{equation}
\end{proposition}

\begin{proof}
Assume that $p \mathbf{u}=q \mathbf{u}=\mathbf{0}$. Since $p$ and $q$ are relative prime, there exist polynomials $a, b$ for which $a\, p +b \, q=1$. Since, for all $n$, $\prodint{p \mathbf{u}, a\, x^n}=\prodint{q \mathbf{u}, b\, x^n}=0$, we see that
$$
\prodint{\mathbf{u}, x^n}=\prodint{\mathbf{u}, (a\, p+b \, q) x^n}=0.
$$
The converse is obvious.
\end{proof}

\subsection{Orthogonal polynomial sequences}

\begin{definition}\label{def-OP}
Let ${\bf u}\in\mathcal{P}'$ and $\{P_n\}_{n\geq0}$ a sequence in $\mathcal{P}$.
\begin{enumerate}
\item[{\rm (i)}]
$\{P_n\}_{n\geq0}$ is called an {\sl orthogonal polynomial sequence (OP\footnote{For abbreviation, we continue to write OP for orthogonal polynomial sequence.})} with respect to ${\bf u}$
if $\{P_n\}_{n\geq0}$ is a simple set (so that $\deg P_n=n$ for all $n$)
and there exists a sequence $\{h_n\}_{n\geq0}$, with $h_n\in\mathbb{C}\setminus\{0\}$, such that
$$
\langle {\bf u},P_mP_n \rangle=h_n\delta_{m,n}\;,\quad m,n=0,1,2,\ldots\;;
$$
\item[{\rm (ii)}]
${\bf u}$ is called {\sl regular} (or {\sl quasi-definite})
if there exists an OP with respect to ${\bf u}$.
\end{enumerate}
\end{definition}

As usual, denoting by $u_j=\langle{\bf u},x^j\rangle$, $j\in\mathbb{N}_0$ the moments of ${\bf u}$,
we define the associated {\sl Hankel determinant} $H_n\equiv H_n({\bf u})$ as
\begin{equation}\label{def-Hn}
H_{-1}=1\;,\quad
H_n=\det\big\{[u_{i+j}]_{i,j=0}^n\big\}\;,\quad n\in\mathbb{N}_0\;.
\end{equation}
It is well known that, given ${\bf u}\in\mathcal{P}'$, then
${\bf u}$ is regular if and only if
\begin{equation}\label{Hn-not-zero}
H_n\neq0\;,\quad \forall n\in\mathbb{N}_0\,.
\end{equation}

One of the most important characterizations of OP relies upon the fact that any
three consecutive polynomials are connected by a very simple relation,
expressed as a three-term recurrence relation (TTRR).

\begin{theorem}\label{Favard}
Let $\{\beta_n\}_{n\geq0}$ and $\{\gamma_n\}_{n\geq0}$ be two arbitrary
sequences of complex numbers, and let $\{P_n\}_{n\geq0}$ be a sequence of (monic) polynomials
defined by the three-term recurrence relation
\begin{equation}\label{TTRRFavard}
P_{n+1}(x)=(x-\beta_n)P_n(x)-\gamma_nP_{n-1}(x)\;,\quad n=0,1,2,\dots,
\end{equation}
with initial conditions $P_{-1}(x)=0$ and $P_{0}(x)=1$.
Then there exists a unique functional ${\bf u}\in\mathcal{P}'$ such that
\begin{equation}\label{Fv1}
\langle{\bf u},1\rangle=u_0=\gamma_0\;,\qquad
\langle{\bf u},P_nP_m\rangle=0\quad\mbox{\rm if}\quad n\neq m\quad (n,m\in\mathbb{N}_0)\;.
\end{equation}
Moreover, ${\bf u}$ is regular and $\{P_n\}_{n\geq0}$ is the corresponding monic OP
if and only if $\gamma_n\neq0$ for each $n\in\mathbb{N}_0$.
\end{theorem}

\begin{remark}
Note the relations
\begin{equation}\label{TTRRb}
\beta_n=\frac{\langle{\bf u},xP_n^2\rangle}{\langle{\bf u},P_n^2\rangle}\;,\quad
\gamma_{n+1}=\frac{\langle{\bf u},P_{n+1}^2\rangle}{\langle{\bf u},P_{n}^2\rangle}
=\frac{H_{n-1}H_{n+1}}{H_{n}^2}\;,\quad n=0,1,\ldots\;.
\end{equation}
\end{remark}

\subsection{Orthogonal polynomials and dual basis}

Since every OP is a simple set of polynomials, it has an
associated dual basis in $\mathcal{P}^\prime$.

\begin{theorem}\label{dualR}
Let ${\bf u}\in\mathcal{P}'$ be regular, $\{P_n\}_{n\geq0}$ the
corresponding monic OP, and $\{{\bf a}_n\}_{n\geq0}$ the associated dual basis.
Then:
\begin{itemize}
\item[{\rm (i)}] For each $n\in\mathbb{N}_0$, ${\bf a}_n$ is explicitly given by
$$
{\bf a}_n=\frac{P_n}{\langle {\bf u},P_n^2\rangle}\, {\bf u}\;.
$$
As a consequence, $\{P_n\}_{n\geq0}$ is a monic OP with respect to ${\bf a}_0$, being
$$
{\bf u}=u_0\,{\bf a}_0\;.
$$
\item[{\rm (ii)}] Let ${\bf v}\in\mathcal{P}'$ and $N\in \mathbb{N}_0$ such that
$$
\langle {\bf v},P_n\rangle=0 \;\;\mbox{\rm if}\;\; n\geq N+1\;.
$$
Then,
$$
{\bf v}=\sum_{j=0}^{N}\langle {\bf v},P_j\rangle\, {\bf a}_j=\phi\,{\bf u}\; ,\quad
\phi(x)=\sum_{j=0}^{N}\frac{\langle {\bf v},P_j\rangle}{\langle{\bf u},P_j^2\rangle}\, P_j(x)\;.
$$
Further, $\deg\phi\leq N$, and $\deg\phi=N$ if and only if $\langle {\bf v},P_N\rangle\neq0$.\medskip
\item[{\rm (iii)}]
Let the TTRR fulfilled by $\{P_n\}_{n\geq0}$ be $(\ref{TTRRFavard})$.
Then $\{{\bf a}_n\}_{n\geq0}$ fulfills
$$
x\,{\bf a}_{n}={\bf a}_{n-1}+\beta_n\,{\bf a}_n+\gamma_{n+1}\,{\bf a}_{n+1} \; , \quad n\in\mathbb{N}_0\;,
$$
with initial conditions ${\bf a}_{-1}={\bf 0}$ and ${\bf a}_{0}=u_0^{-1}\,{\bf u}$.
\end{itemize}
\end{theorem}

\begin{corollary}\label{OPwrtv}
Let $\{P_n\}_{n\geq0}$ be a monic OP (with respect to some functional in
$\mathcal{P}'$) and let ${\bf v}\in\mathcal{P}'$. Then $\{P_n\}_{n\geq0}$
is a monic OP with respect to ${\bf v}$ if and only if
\begin{equation}
\langle{\bf v},1\rangle\neq0 \;, \qquad \langle{\bf v},P_n\rangle=0\;,\quad n=0,1,2\ldots\; \, .
\label{uSPOMregular}
\end{equation}
\end{corollary}

\begin{theorem}\label{OPequiv}
Under the hypothesis of Proposition \ref{uv-equiv}, assume further that
$\{P_n\}_{n\geq0}$ is a monic OP with respect to the functional ${\bf u}\in\mathcal{P}'$,
and let the TTRR fulfilled by $\{P_n\}_{n\geq0}$ be $(\ref{TTRRFavard})$.
Then, $\{Q_n\}_{n\geq0}$ is a monic OP with respect to 
\begin{equation}\label{Qn-afim-OP}
{\bf v}=\big(\bm{h}_{a^{-1}}\circ\bm{\tau}_{-b}\big){\bf u}\;,
\end{equation}
and the TTRR fulfilled by $\{Q_n\}_{n\geq0}$ is
\begin{equation}\label{ttrr-afim-Qn}
xQ_n(x)=Q_{n+1}(x)+\widehat{\beta}_nQ_n(x)+\widehat{\gamma}_nQ_{n-1}(x)\;,\quad n=0,1,2\ldots\;,
\end{equation}
with initial conditions $Q_{-1}(x)=0$ and $Q_0(x)=1$, where
\begin{equation}\label{beta-gamma-ttrr-afim-Qn}
\widehat{\beta}_n=\frac{\beta_n-b}{a}\;,\quad \widehat{\gamma}_n=\frac{\gamma_n}{a^2}\;.
\end{equation}
\end{theorem}

\section{Distributional differential equation}

The distributional differential equation has the form
\begin{equation}\label{Pearson-DEq}
D(\phi{\bf u})=\psi{\bf u}\;,
\end{equation}
where $\phi\in\mathcal{P}_2$ and $\psi\in\mathcal{P}_1$,
and ${\bf u}\in\mathcal{P}'$ is the unknown.
Notice that we do not require \emph{a priori} ${\bf u}$ to be a regular functional.
We may write
\begin{equation}\label{Pearson-PhiPsi}
\phi(x)= ax^2+bx+c\;,\quad \psi(x)= px+q\;,
\end{equation}
being $a,b,c,p,q\in\mathbb{C}$.
We also define, for each integer or rational number $n$,
\begin{equation}\label{psi-n}
\psi_n=\psi+n\phi^{\prime}\;, \quad
d_n=\psi_{n/2}^{\prime}=na+p\;, \quad e_n=\psi_n(0)=nb+q\;.
\end{equation}
Notice that $\psi_n(x)=d_{2n}x+e_n\in\mathcal{P}_1$.
Finally, for each ${\bf u}\in\mathcal{P}'$ and each $n\in\mathbb{N}_0$, we set
\begin{equation}\label{Phi-n-u}
{\bf u}^{[n]}=\phi^n{\bf u}\;.
\end{equation}

We begin with the following elementary result.

\begin{lemma}\label{Pearson-lemma1} 
Let ${\bf u}\in\mathcal{P}'$.
Then ${\bf u}$ satisfies the distributional differential equation $(\ref{Pearson-DEq})$
if and only if the corresponding sequence of moments, $u_n=\langle{\bf u},x^n\rangle$,
satisfies the second order linear difference equation 
\begin{equation}\label{le1a}
d_nu_{n+1}+e_nu_n+n\phi(0)u_{n-1}=0\;, \quad n\in\mathbb{N}_0\; .
\end{equation}
Moreover, if ${\bf u}$ satisfies $(\ref{Pearson-DEq})$, then ${\bf u}^{[n]}$ satisfies
\begin{equation}\label{le1b}
D\big(\phi{\bf u}^{[n]}\big)=\psi_n{\bf u}^{[n]}\;, \quad n\in\mathbb{N}_0\;.
\end{equation}
\end{lemma}

Notice that if both $\phi$ and $\psi$ vanish identically 
then (\ref{Pearson-DEq}) reduces to a trivial equation,
so we will exclude this situation from our study.

\begin{lemma}\label{poly-grau21}
Let ${\bf u}\in\mathcal{P}'$. Suppose that ${\bf u}$ is regular and satisfies
$(\ref{Pearson-DEq})$, being
$\phi\in\mathcal{P}_2$ and $\psi\in\mathcal{P}_1$,
and assume that at least one of the polynomials $\phi$ and $\psi$ is nonzero.
Then neither $\phi$ nor $\psi$ is the zero polynomial, and
\begin{equation}\label{grauPsi1}
\deg\psi=1\;.
\end{equation}
\end{lemma}

Given a monic polynomial $P_n$ of degree $n$ (which needs not to belong to an OP),
we denote by $P_n^{[k]}$ the monic polynomial of degree $n$ defined by
\begin{equation}\label{Pnk-deriv}
P_n^{[k]}(x)=\frac{{\rm d}^k}{{\rm d}x^k}\,\frac{P_{n+k}(x)}{(n+1)_k}, \quad k,n\in\mathbb{N}_0\;,
\end{equation}
where, for a given $\alpha\in\mathbb{C}$, $(\alpha)_n$ is the {\it Pochhammer symbol}, defined as
\begin{equation}\label{Pochhammer}
(\alpha)_0=1\;,\qquad (\alpha)_n=\alpha(\alpha+1)\cdots(\alpha+n-1)\;,\quad n\in\mathbb{N}\;.
\end{equation}
Clearly, if $\{P_n\}_{n\geq0}$ is a simple set in $\mathcal{P}$, then so is $\{P_n^{[k]}\}_{n\geq0}$.
Under such conditions, there is a beautiful relation between the associated dual basis:
\begin{equation}\label{rel-dual-an-ank}
D^{k}\big(\,{\bf a}_n^{[k]}\,\big)=(-1)^k(n+1)_k\,{\bf a}_{n+k} \;,\quad k,n\in\mathbb{N}_0 \,,
\end{equation}
where $\{{\bf a}_n\}_{n\geq0}$ and $\big\{{\bf a}_n^{[k]}\,\big\}_{n\geq0}$ are the dual basis
in $\mathcal{P}'$ associated with $\{P_n\}_{n\geq0}$ and $\{P_n^{[k]}\}_{n\geq0}$, respectively.

\begin{lemma}\label{reg-lemma3}
Let ${\bf u}\in\mathcal{P}'$, and suppose that
${\bf u}$ satisfies the distributional differential equation $(\ref{Pearson-DEq})$,
with $\phi$ and $\psi$ given by $(\ref{Pearson-PhiPsi})$,
being at least one of these polynomials nonzero.
Suppose further that ${\bf u}$ is regular. Then
\begin{equation}\label{r1}
d_n=na+p\neq0\; , \quad \forall n\in\mathbb{N}_0\; .
\end{equation}
Moreover, if $\{P_n\}_{n\geq0}$ denotes the monic OP with respect to ${\bf u}$, and
$P_n^{[k]}$ is defined by $(\ref{Pnk-deriv})$, then ${\bf u}^{[k]}=\phi^k{\bf u}$ is regular and
$\{P_n^{[k]}\}_{n\geq0}$ is its monic OP, for each $k\in\mathbb{N}_0$.
\end{lemma}

We may now establish necessary and sufficient conditions ensuring
the regularity of a given functional ${\bf u}\in\mathcal{P}'$ satisfying (\ref{Pearson-PhiPsi}). The next results was proved in \cite[Theorem 2]{PacoPetronilho1994-C7} (see also \cite{CMP} in a more general context).

\begin{theorem}\label{vCOP-Thm1}
Let ${\bf u}\in\mathcal{P}^\prime\setminus\{{\bf 0}\}$,
and suppose that ${\bf u}$ satisfies
\begin{equation}\label{EqDistC1}
D(\phi{\bf u})=\psi{\bf u}\;,
\end{equation}
where $\phi$ and $\psi$ are nonzero polynomials such that
$\phi\in\mathcal{P}_2$ and $\psi\in\mathcal{P}_1$. Set
$$
\phi(x)=ax^2+bx+c\,,\quad\psi(x)=px+q\,,\quad
d_n=na+p\,,\quad e_n=nb+q \quad(n\in\mathbb{N}_0)\,.
$$
Then, ${\bf u}$ is regular if and only if 
\begin{equation}\label{EqDistC2}
d_n\neq0\,,\quad \phi\Big(-\frac{e_n}{d_{2n}}\Big)\neq0\,,\quad
\forall n\in\mathbb{N}_0\;. 
\end{equation}
Moreover, under these conditions, the monic OP $\{P_n\}_{n\geq0}$
with respect to ${\bf u}$ is given by the three-term recurrence relation
\begin{equation}\label{ttrrC1}
P_{n+1}(x)=(x-\beta_n)P_n(x)-\gamma_nP_{n-1}(x)\;,\quad  n\in\mathbb{N}_0 
\end{equation}
with initial conditions $P_{-1}(x)=0$ and $P_0(x)=1$, being
\begin{equation}\label{EqDistC3}
\beta_n=\frac{ne_{n-1}}{d_{2n-2}}-\frac{(n+1)e_{n}}{d_{2n}}\, ,\quad
\gamma_{n+1}=-\frac{(n+1)d_{n-1}}{d_{2n-1}d_{2n+1}}\phi\Big(-\frac{e_n}{d_{2n}}\Big)
\;,\quad  n\in\mathbb{N}_0\;.
\end{equation}
In addition, for each $n\in\mathbb{N}_0$, $P_n$ satisfies the distributional Rodrigues formula
\begin{equation}\label{EqDistRod}
P_n{\bf u}=k_n\,D^n\big(\phi^n{\bf u}\big)\;,\quad
k_n=\prod_{i=0}^{n-1}d_{n+i-1}^{-1}\,.
\end{equation}
\end{theorem}

\section{Classical orthogonal polynomials}

The \emph{classical functionals} are the regular solutions (in $\mathcal{P}'$) of the distributional equation \eqref{Pearson-DEq}. The corresponding OP are called \emph{classical orthogonal polynomials}.
In this section we present the most significant results concerning this important class of OP.

\subsection{Definition and characterizations}

\begin{definition}\label{def-u-classical}
Let ${\bf u}\in\mathcal{P}'$.
${\bf u}$ is called a {\sl classical} functional if the following two conditions hold:
\begin{enumerate}
\item[{\rm (i)}]
${\bf u}$ is regular;
\item[{\rm (ii)}]
${\bf u}$ satisfies the distributional differential equation
\begin{equation}\label{EDClassic1}
D(\phi{\bf u})=\psi{\bf u}\;,
\end{equation}
where $\phi$ and $\psi$ are polynomials fulfilling
\begin{equation}\label{grauPhiPsi}
\deg\phi\leq2\;,\quad\deg\psi=1\;.
\end{equation}
\end{enumerate}
An OP $\{P_n\}_{n\geq0}$ with respect to a classical functional
is called a {\sl classical OP}.
\end{definition}

\begin{remark}\em
According with Lemma \ref{poly-grau21}, in the above definition
conditions (\ref{grauPhiPsi}) may be replaced by the weaker conditions
\begin{equation}\label{phiP2psiP1}
\phi\in\mathcal{P}_2\;,\quad\psi\in\mathcal{P}_1\;,\quad\{\phi,\psi\}\neq\mathcal{P}_{-1}=\{0\}\;.
\end{equation}
\end{remark}

Theorem \ref{vCOP-Thm1} gives necessary and sufficient conditions for the existence of solutions
of the distributional equation \eqref{Pearson-DEq},
characterizing also such functionals (and, in particular, solving the question of the existence of classical functionals).
Thus, we may state:
{\it a functional ${\bf u}\in\mathcal{P}'\setminus\{{\bf 0}\}$ is classical if and only if
there exist $\phi\in\mathcal{P}_2$ and $\psi\in\mathcal{P}_1$ such that the following conditions hold:
\begin{equation}\label{P-regular1}
\begin{array}{rl}
{\rm (i)} & D(\phi{\bf u})=\psi{\bf u}\,; \\ 
{\rm (ii)} & na+p\neq0\;,\quad\displaystyle\phi\left(-\frac{nb+q}{2na+p}\right)\neq0\;,
\quad \forall n\in\mathbb{N}_0\;,
\end{array}
\end{equation}
where we have set $\phi(x)=ax^2+bx+c$ and $\psi(x)=px+q$.}

In the next proposition we state several characterizations of the classical OP.
For convenience, we introduce the concept of admissible pair of polynomials.

\begin{definition}
$(\phi,\psi)$ is called an {\sl admissible pair} if
$$
\phi\in\mathcal{P}_2\; , \quad \psi\in\mathcal{P}_1 \; ,\quad
d_n=\psi^{\prime}+\mbox{$\frac{n}{2}$}\,\phi^{\prime\prime}\neq0 \; , \;\;\forall n\in\mathbb{N}_0.
$$
\end{definition}
Introducing this concept makes sense, since according with 
conditions (ii) in (\ref{P-regular1}), only
admissible pairs may appear in the framework of the theory of classical OP.

\begin{theorem}[characterizations of the classical OP]\label{ThmClassicalOP}
Let ${\bf u}\in\mathcal{P}^\prime$, regular, and let $\{P_n\}_{n\geq0}$ be its monic OP.
Then the following properties are equivalent:
\begin{enumerate}
\item[{\rm C1.}]
${\bf u}$ is classical, i.e., there are nonzero polynomials $\phi\in\mathcal{P}_2$ and $\psi\in\mathcal{P}_1$
such that ${\bf u}$ satisfies
$$D(\phi {\bf u})=\psi {\bf u}\; ;$$
\item[{\rm C1$'$.}]
there is an admissible pair $(\phi,\psi)$ such that ${\bf u}$ satisfies
$$D(\phi {\bf u})=\psi {\bf u}\; ;$$
\item[{\rm C2.}] {\rm (Al-Salam $\&$ Chihara)} 
there exist a polynomial $\phi\in\mathcal{P}_2$ and, for each $n\in\mathbb{N}_0$,
complex parameters $a_n$, $b_n$ and $c_n$, with $c_n\neq0$ if $n\geq1$, such that
$$
\phi(x)P_n^{\prime}(x)=a_nP_{n+1}(x)+b_nP_n(x)+c_nP_{n-1}(x)\;,\quad n\geq0 \; ;
$$
\item[{\rm C3.}] {\rm (Hahn)} 
$\Big\{ P_n^{[k]}=\frac{{\rm d}^k}{{\rm d}x^k}\frac{P_{n+k}}{(n+1)_k}\Big\}_{n\geq0}$ is a monic OP for some $k\in\mathbb{N}\,$;
\item[{\rm C3$'$.}] 
$\big\{ P_n^{[k]}\big\}_{n\geq0}$ is a monic OP for each $k\in\mathbb{N}\,$;
\item[{\rm C4.}]
there exist $k\in\mathbb{N}$ and complex parameters $r_{n}^{[k]}$ and $s_{n}^{[k]}$ such that
\begin{align}\label{Eaux}
P_n^{[k-1]}(x)=P_n^{[k]}(x)+r_{n}^{[k]}P_{n-1}^{[k]}(x)+s_{n}^{[k]}P_{n-2}^{[k]}(x) \; , \quad n\geq2\;;
\end{align}
\item[{\rm C4$'$.}]\footnote{It was proved for $k=1$ by Geronimus (see \cite[(42)]{Geronimus1940}).}
for each $k\in\mathbb{N}$, there exist parameters $r_{n}^{[k]}$ and $s_{n}^{[k]}$ such that \eqref{Eaux} holds;
\item[{\rm C5.}] {\rm (Bochner)} 
there exist polynomials $\phi$ and $\psi$ and, for each $n\geq0$, a complex parameter $\lambda_n$,
with $\lambda_n\neq0$ if $n\geq1$, such that $y=P_n(x)$ is a solution of the second order
ordinary differential equation
$$
\phi(x) y^{\prime\prime} + \psi(x) y^{\prime} +\lambda_n y =0 \;,\quad n\geq0\;;
$$
\item[{\rm C6.}] {\rm (Maroni)} 
there is an admissible pair $(\phi,\psi)$ so that the formal Stieltjes series associated with ${\bf u}$,
$S_{\bf u}(z)=-\sum_{n=0}^{\infty}u_{n}/z^{n+1}$,
satisfies (formally) 
$$
\phi(z)S_{\bf u}^{\prime}(z)=[\psi(z)-\phi^{\prime}(z)]S_{\bf u}(z)+(\psi^{\prime}
-\mbox{$\frac12$}\,\phi^{\prime \prime}) u_0 \; ;
$$
\item[{\rm C7.}] {\rm (McCarthy)} 
there exists an admissible pair $(\phi,\psi)$ and, for each $n\geq1$, complex parameters $h_n$ and $t_n$ such that
$$
\phi(P_nP_{n-1})^{\prime}(x)=h_nP_n^2(x)-(\psi-\phi^{\prime})P_nP_{n-1}(x)+t_nP_{n-1}^2(x) \; ;
$$
\item[{\rm C8.}] {\rm (distributional Rodrigues formula)} 
there exist a polynomial $\phi\in\mathcal{P}_2$ and nonzero complex parameters $k_n$
such that
$$
P_n(x) {\bf u} = k_n D^n\big(\phi^n(x){\bf u}\big) \; , \quad n\geq0 \;.
$$
\end{enumerate}
Moreover, the polynomials $\phi$ and $\psi$ 
may be taken the same in all properties above where they appear.
In addition, let the TTRR 
fulfilled by the monic OP $\{P_n\}_{n\geq0}$ be
$$
P_{n+1}(x)=(x-\beta_n)P_n(x)-\gamma_{n}P_{n-1}(x)\;,\quad n\geq0
$$
($P_{-1}(x)=0$; $P_0(x)=1$).
Write $\phi(x)=ax^2+bx+c$, $\psi(x)=px+q$, $d_n=na+p$, and $e_n=nb+q$. Then
$$
\beta_n=-\frac{d_{-2}q+2bnd_{n-1}}{d_{2n}d_{2n-2}}\;,\quad
\gamma_{n}=-\frac{nd_{n-2}}{d_{2n-3}d_{2n-1}}\phi\Big(-\frac{e_{n-1}}{d_{2n-2}}\Big)\;,
$$
and the parameters appearing in the above characterizations
may be computed explicitly: 
\begin{align*}
a_n&=na, \quad  \; \;\qquad b_n=-\mbox{$\frac12$}\psi(\beta_n), \quad \quad c_n=-d_{n-1}\gamma_n,\\[0.5pt]
r_n^{[1]}&=\mbox{$\frac12$}\frac{\psi(\beta_n)}{d_{n-1}}, \quad s_n^{[1]}=-\frac{(n-1)a}{d_{n-2}}\gamma_n,\\
\lambda_n&=-nd_{n-1},\quad h_n=d_{2n-3}, \quad t_n=-d_{2n-1}\gamma_n, \quad k_n=\prod_{i=0}^{n-1}d_{n+i-1}^{-1}.
\end{align*}
\end{theorem}

\begin{proof}
By Lemma \ref{reg-lemma3} and Theorem \ref{vCOP-Thm1},
C1$\,\Leftrightarrow\,$C1$'$, C1$\,\Rightarrow\,$C3$'$, and C1$'$$\,\Leftrightarrow\,$C8.
Clearly, C3$'$$\,\Rightarrow\,$C3 and C4$'$$\,\Rightarrow\,$C4.
We show that C3$'$$\,\Rightarrow\,$C4$'$ using the same arguments
of the proof of C3$\,\Rightarrow\,$C4 given in bellow.
The proof of C1$'$$\,\Leftrightarrow\,$C6 is left to the reader.
Thus, we only need to show that:
\smallskip
\begin{center}
C1$'$$\,\Rightarrow\,$C2$\,\Rightarrow\,$C3$\,\Rightarrow\,$C4$\,\Rightarrow\,$C1$\,,\;\;$
C1$\,\Leftrightarrow\,$C5$\,,\;\;$ C2$\,\Leftrightarrow\,$C7.
\end{center}
\smallskip

(C1$'$$\,\Rightarrow\,$C2).
Assume that C1$'$ holds. Fix $n\in\mathbb{N}_0$. Since $\deg(\phi P_n')\leq n+1$, then
\begin{equation}\label{ThmChar1}
\phi P'_n=\sum_{j=0}^{n+1}a_{n,j}P_j\,,\quad
a_{n,j}=\frac{\langle{\bf u},\phi P_n'P_j\rangle}{\langle{\bf u},P_j^2\rangle}\;.
\end{equation}
For each integer number $j$, with $0\leq j\leq n+1$, we deduce
\begin{equation}\label{ThmChar2}
\begin{array}{rl}
\langle{\bf u},\phi P_n'P_j\rangle &=\,\langle\phi{\bf u},(P_nP_j)'-P_nP_j'\rangle=
-\langle D(\phi{\bf u}),P_nP_j\rangle-\langle \phi{\bf u}, P_nP_j'\rangle \\ [0.5em]
&=\, -\langle {\bf u},\psi P_jP_n\rangle-\langle {\bf u},\phi P_j'P_n\rangle\;.
\end{array}
\end{equation}
If $0\leq j\leq n-2$ we obtain $\langle{\bf u},\phi P_n'P_j\rangle=0$, and so $a_{n,j}=0$.
Thus, (\ref{ThmChar1}) reduces to
$$
\phi P'_n=a_nP_{n+1}+b_nP_n+c_nP_{n-1}\;,\quad n\geq0\;,
$$
where, writing $\phi(x)=ax^2+bx+c$ and $\psi(x)=px+q$,
$a_n=na$ (by comparison of coefficients), $b_n=a_{n,n}$, and $c_n=a_{n,n-1}$.
Setting $j=n-1$ in (\ref{ThmChar2}), we deduce
$$
\langle{\bf u},\phi P_n'P_{n-1}\rangle=
-\langle {\bf u},(\psi P_{n-1}+\phi P_{n-1}')P_n\rangle=-d_{n-1}\langle{\bf u}, P_n^2\rangle\;,
$$
hence
$$
c_n=a_{n,n-1}=\frac{\langle{\bf u},\phi P_n'P_{n-1}\rangle}{\langle{\bf u},P_{n-1}^2\rangle}
=\frac{\langle{\bf u},\phi P_n'P_{n-1}\rangle}{\langle{\bf u},P_{n}^2\rangle}
\frac{\langle{\bf u},P_n^2\rangle}{\langle{\bf u},P_{n-1}^2\rangle}=-d_{n-1}\gamma_n\;,\quad n\geq1\;.
$$
Since, by hypothesis, $(\phi,\psi)$ is an admissible pair,
then we may conclude that $c_n\neq0$ for each $n\geq1$. Thus C1$'$$\,\Rightarrow\,$C2.
Notice that taking $j=n$ in (\ref{ThmChar2}) yields
$$
\langle{\bf u},\phi P_n'P_n\rangle=-\mbox{$\frac12\,$}\langle {\bf u},\psi P_n^2\rangle
=-\mbox{$\frac12\,$}\big(p\langle {\bf u},x P_n^2\rangle+q\langle {\bf u},P_n^2\rangle\big)\;,
$$
hence we deduce the expression for $b_n$ given in the statement of the theorem:
$$
b_n=a_{n,n}=\frac{\langle{\bf u},\phi P_n'P_{n}\rangle}{\langle{\bf u},P_{n}^2\rangle}
=-\mbox{$\frac12\,$}\Big(p\frac{\langle {\bf u},x P_n^2\rangle}{\langle {\bf u},P_n^2\rangle}+q\Big)
=-\mbox{$\frac12$}\psi(\beta_n)\;.
$$

(C2$\,\Rightarrow\,$C3).
Suppose that C2 holds.
We will show that $\{ P_n^{[1]}=P_{n+1}^{\prime}/(n+1)\}_{n\geq0}$
is a monic OP with respect to ${\bf v}=\phi{\bf u}$.
Indeed, for each $n\in\mathbb{N}_0$ and $0\leq m\leq n$, 
\begin{align*}
(n+1)\langle{\bf v},x^m P_n^{[1]} \rangle &=\langle\phi{\bf u},x^mP_{n+1}'\rangle
=\langle{\bf u},\big(\phi P_{n+1}'\big)x^m\rangle \\
&= \langle {\bf u},(a_{n+1}P_{n+2}+b_{n+1}P_{n+1}+c_{n+1}P_{n})x^m\rangle\\
&=c_{n+1}\langle {\bf u},P_{n}^2\rangle\delta_{m,n}\;. 
\end{align*}
Therefore, since $c_{n+1}\neq0$ for each $n\geq0$,
we conclude that $\{P_n^{[1]}\}_{n\geq0}$
is a monic OP (with respect to ${\bf v}=\phi{\bf u}$).
\smallskip

(C3$\,\Rightarrow\,$C4).
By hypothesis, $\{P_n^{[k]}=\frac{{\rm d}^k}{{\rm d}x^k}\big(\frac{P_{n+k}}{(n+1)_k}\big)\}_{n\geq0}$
is a monic OP for some (fixed) $k\in\mathbb{N}$. Then there exists $\beta_n^{[k]}\in\mathbb{C}$ and $\gamma_n^{[k]}\in\mathbb{C}\setminus\{0\}$ such that
\begin{equation}\label{PnkTTRR}
xP_n^{[k]}=P_{n+1}^{[k]}+\beta_n^{[k]}P_n^{[k]}+\gamma_n^{[k]}P_{n-1}^{[k]} \;,\quad n\in\mathbb{N}_0 \,.
\end{equation}
Similarly, there exists $\beta_n\in\mathbb{C}$ and $\gamma_n\in\mathbb{C}\setminus\{0\}$ such that
\begin{equation}\label{PnTTRR}
xP_n=P_{n+1}+\beta_nP_n+\gamma_{n}P_{n-1}\;,\quad n\in\mathbb{N}_0\,.
\end{equation}
Changing $n$ into $n+k$ in (\ref{PnTTRR}), then taking the derivative of order $k$ in both sides of the resulting equation
and using Leibnitz rule on the left-hand side, we find 
$$
xP_n^{[k]}+\frac{k}{n+1}P_{n+1}^{[k-1]}=
\frac{n+k+1}{n+1}P_{n+1}^{[k]}+\beta_{n+k}P_n^{[k]}+\frac{n\gamma_{n+k}}{n+k}P_{n-1}^{[k]} \;,\quad n\in\mathbb{N}_0 \,.
$$
In this equation, replacing $xP_n^{[k]}$ by the right-hand side of (\ref{PnkTTRR}), and then changing $n$ into $n-1$,
we obtain \eqref{Eaux}, with
$$
r_{n}^{[k]}=\frac{n\,\big(\beta_{n+k-1}-\beta_{n-1}^{[k]}\big)}{k}\;,\quad
s_{n}^{[k]}=\frac{n\,\big((n-1)\gamma_{n+k-1}-(n+k-1)\gamma_{n-1}^{[k]}\big)}{k(n+k-1)}\;.
$$
\smallskip

(C4$\,\Rightarrow\,$C1).
By hypothesis \eqref{Eaux} holds.
Let $\{{\bf a}_n\}_{n\geq0}$ and $\{{\bf a}_n^{[k]}\}_{n\geq0}$ be the dual basis
for $\{P_n\}_{n\geq0}$ and $\{P_n^{[k]}\}_{n\geq0}$, respectively. By Proposition \ref{expDB1},
$\,{\bf a}_n^{[k]}=\sum_{j\geq0}\langle{\bf a}_n^{[k]},P_j^{[k-1]}\rangle {\bf a}_j^{[k-1]}$ for each $n\in\mathbb{N}_0$.
Using \eqref{Eaux}, we compute
\begin{align*}
\langle{\bf a}_n^{[k]},P_j^{[k-1]}\rangle
&=\langle{\bf a}_n^{[k]},P_j^{[k]}\rangle
+r_j^{[k]}\langle{\bf a}_n^{[k]},P_{j-1}^{[k]}\rangle
+s_j^{[k]}\langle{\bf a}_n^{[k]},P_{j-2}^{[k]}\rangle\\
&=\left\{
\begin{array}{cl}1\;, &\mbox{\rm if}\; j=n \\ [0.5em]
r_{n+1}^{[k]}\;, &\mbox{\rm if}\; j=n+1 \\ [0.5em]
s_{n+2}^{[k]}\;, &\mbox{\rm if}\; j=n+2 \\ [0.5em]
0\;, &\mbox{\rm otherwise}\,. 
\end{array}
\right.
\end{align*}
Hence
$$
{\bf a}_n^{[k]}= {\bf a}_n^{[k-1]}+r_{n+1}^{[k]}{\bf a}_{n+1}^{[k-1]}+s_{n+2}^{[k]}{\bf a}_{n+2}^{[k-1]}\;,\quad n\in\mathbb{N}_0\;.
$$
Taking the (distributional) derivative of order $k$ in both sides of this equation, and using the relations
$D^j\big({\bf a}_n^{[j]}\big)=(-1)^j(n+1)_j\,{\bf a}_{n+j}$, we obtain
$$
D\left(\frac{1}{n+k}\,{\bf a}_{n+k-1}+\frac{r_{n+1}^{[k]}}{n+1}\,{\bf a}_{n+k}
+\frac{(n+k+1)s_{n+2}^{[k]}}{(n+1)(n+2)}\,{\bf a}_{n+k+1}\right)=-{\bf a}_{n+k}\;,
\; n\in\mathbb{N}_0\;.
$$
Therefore, since, by Theorem \ref{dualR}, ${\bf a}_j=\frac{P_j}{\langle{\bf u},P_j^2\rangle}\,{\bf u}$
for $j\in\mathbb{N}_0$ and, by (\ref{TTRRb}),
$\gamma_j=\frac{{\langle{\bf u},P_j^2\rangle}}{{\langle{\bf u},P_{j-1}^2\rangle}}$ for $j\in\mathbb{N}$,
being $\gamma_j$ the parameter appearing in (\ref{PnTTRR}), we deduce
\begin{equation}\label{DPhink1}
D\big(\Phi_{n+k+1}\,{\bf u}\big)=-P_{n+k}\,{\bf u}\;,\quad n\in\mathbb{N}_0\;,
\end{equation}
where $\Phi_{n+k+1}$ is a polynomial of degree at most $n+k+1$, given by
\begin{align*}
&\Phi_{n+k+1}(x)\\
&=\frac{\gamma_{n+k}}{n+k}\,P_{n+k-1}(x)+
\frac{r_{n+1}^{[k]}}{n+1}\,P_{n+k}(x)+\frac{(n+k+1)s_{n+2}^{[k]}}{(n+1)(n+2)\gamma_{n+k+1}}\,P_{n+k+1}(x)\,.
\end{align*}
Since $\Phi_{n+k+1}$ is a (finite) linear combination of polynomials of the simple set $\{P_j\}_{j\geq0}$
and $\gamma_{n+k}\neq0$, then $\Phi_{n+k+1}$ does not vanishes identically, so
$\Phi_{n+k+1}\in\mathcal{P}_{n+k+1}\setminus\{0\}$.
Setting $n=0$ and $n=1$ in (\ref{DPhink1}) we obtain the two equations
\begin{equation}\label{DPhink1a}
D\big(\Phi_{k+1}\,{\bf u}\big)=-P_{k}\,{\bf u}\;,\quad
D\big(\Phi_{k+2}\,{\bf u}\big)=-P_{k+1}\,{\bf u}\;.
\end{equation}
If $k=1$ it follows immediately from the first of these equations that C1 holds. 
Henceforth, assume that $k\geq2$.
Setting $n=0$ and $n=1$ in the definition of $\Phi_{n+k+1}$
and using the TTRR (\ref{PnTTRR}), we easily deduce
\begin{equation}\label{DPhink1b}
\left\{
\begin{array}{rcl}
\Phi_{k+1}(x) &=& E_0(x;k)P_{k+1}(x)+F_1(x;k)P_{k}(x)\;, \\ [0.5em]
\Phi_{k+2}(x) &=& G_1(x;k)P_{k+1}(x)+H_0(x;k)P_{k}(x)\;,
\end{array}
\right.
\end{equation}
where $E_0(\cdot;k),H_0(\cdot;k)\in\mathcal{P}_0$ and $F_1(\cdot;k),G_1(\cdot;k)\in\mathcal{P}_1$,
explicitly given by
\begin{equation}\label{DPhink1ba}
\begin{array}{l}
E_0(x;k) = \displaystyle\frac{(k+1)s_{2}^{[k]}}{2\gamma_{k+1}}-\frac{1}{k}\;,\quad
F_1(x;k) = \frac{x-\beta_k}{k}+r_{1}^{[k]} \;, \\ [0.75em]
G_1(x;k) = \displaystyle\frac{(k+2)s_{3}^{[k]}(x-\beta_{k+1})}{6\gamma_{k+2}}+\frac{r_{2}^{[k]}}{2}\;, \;
H_0(x;k) = \frac{\gamma_{k+1}}{k+1}-\frac{(k+2)s_{3}^{[k]}\gamma_{k+1}}{6\gamma_{k+2}}\;.
\end{array}
\end{equation}
Let $\Delta_2(x)\equiv\Delta_2(x;k)=E_0(x;k)H_0(x;k)- F_1(x;k)G_1(x;k)$,
the determinant of the system (\ref{DPhink1b}).
Using (\ref{DPhink1a})--(\ref{DPhink1ba}), and taking into account
that ${\bf u}$ is regular, we may prove that
$\Delta_2\in\mathcal{P}_2\setminus\{0\}$ .
Solving (\ref{DPhink1b}) for $P_k$ and $P_{k+1}$ we obtain
\begin{eqnarray}
\label{DPhink1Det2}
\Delta_2(x)P_{k+1}(x)=H_0(x;k)\Phi_{k+1}(x)-F_1(x;k)\Phi_{k+2}(x)\,, \\
\label{DPhink1Det3}
\Delta_2(x)P_{k}(x)=E_0(x;k)\Phi_{k+2}(x)-G_1(x;k)\Phi_{k+1}(x)\;.
\end{eqnarray}
Since $P_k$ and $P_{k+1}$ cannot share zeros,
it follows from (\ref{DPhink1Det2})--(\ref{DPhink1Det3})
that any common zero of $\Phi_{k+1}$ and $\Phi_{k+2}$
(if there is some) must be a zero of $\Delta_2$.
Let $\Phi$ be the greatest common divisor of $\Phi_{k+1}$ and $\Phi_{k+2}$, i.e.,
$$
\Phi(x)=\mbox{g.c.d.}\,\{\Phi_{k+1}(x),\Phi_{k+2}(x)\}\;.
$$
Any zero of $\Phi$ is also a zero of both $\Phi_{k+1}$ and $\Phi_{k+2}$,
and so it is a zero of $\Delta_2$. Therefore, $\Phi\in\mathcal{P}_2\setminus\{0\}$.
(Notice that indeed $\Phi\not\equiv0$, since $\Phi_{k+1}\not\equiv0$ and $\Phi_{k+2}\not\equiv0$.)
Moreover, there exist polynomials $\Phi_{1,k}$ and $\Phi_{2,k}$, with no common zeros, such that
\begin{eqnarray}
\label{DPhink1A}\Phi_{k+1}=\Phi\,\Phi_{1,k}\;,\quad\Phi_{k+2}=\Phi\,\Phi_{2,k}\;,
\qquad\qquad\qquad \\ [0.25em]
\label{DPhink1Aa}
\Phi_{1,k}\in\mathcal{P}_{k+1-\ell}\setminus\{0\}\;,\quad
\Phi_{2,k}\in\mathcal{P}_{k+2-\ell}\setminus\{0\}\;,\quad
\ell=\deg\Phi\leq2\;.
\end{eqnarray}
From (\ref{DPhink1a}) and (\ref{DPhink1A}) we deduce
\begin{equation}\label{DPhink1B}
\Phi_{1,k}D(\Phi{\bf u})=-(P_k+\Phi_{1,k}'\Phi){\bf u}\;,\quad
\Phi_{2,k}D(\Phi{\bf u})=-(P_{k+1}+\Phi_{2,k}'\Phi){\bf u}\;.
\end{equation}
Combining these two equations yields
$\big(\Phi_{1,k}(P_{k+1}+\Phi_{2,k}'\Phi)-\Phi_{2,k}(P_k+\Phi_{1,k}'\Phi)\big){\bf u}={\bf 0}$,
and so, since ${\bf u}$ is regular,
$\Phi_{1,k}(P_{k+1}+\Phi_{2,k}'\Phi)=\Phi_{2,k}(P_k+\Phi_{1,k}'\Phi)$.
Therefore, taking into account that $\Phi_{1,k}$ and $\Phi_{2,k}$ have no common zeros
and (\ref{DPhink1Aa}) holds, we may ensure that there exists a polynomial $\Psi\in\mathcal{P}_1$ such that
\begin{equation}\label{DPhink1E}
P_k+\Phi_{1,k}'\Phi=-\Psi\Phi_{1,k}\;,\quad
P_{k+1}+\Phi_{2,k}'\Phi=-\Psi\Phi_{2,k}\;.
\end{equation}
Combining equations (\ref{DPhink1B}) and (\ref{DPhink1E}) we deduce
$$
\Phi_{1,k}\big(D(\Phi{\bf u})-\Psi{\bf u}\big)=
\Phi_{2,k}\big(D(\Phi{\bf u})-\Psi{\bf u}\big)={\bf 0}\;.
$$
From these equations, and using once again the fact that
$\Phi_{1,k}$ and $\Phi_{2,k}$ have no common zeros,
we conclude, by Proposition \ref{prop-uvzero}, that $D(\Phi{\bf u})=\Psi{\bf u}$. 
Thus C4$\,\Rightarrow\,$C1.
The formulas for $r_n^{[1]}$ and $s_n^{[1]}$ given in the statement of
the theorem may be derived as follows. We have already proved that
C4$\,\Rightarrow\,$C1$\,\Rightarrow\,$C1$'$$\,\Rightarrow\,$C2$\,\Rightarrow\,$C3$\,\Rightarrow\,$C4,
and we see that the polynomials $\phi$ and $\psi$ appearing in all these characterizations may be taken the same.
As we have seen, 
the formulas for $b_n$ and $c_n$ given in the statement of the theorem hold.
We now use these formulas to obtain the expressions for $r_n^{[1]}$ and $s_n^{[1]}$.
Set $Q_n=P_n^{[1]}=P_{n+1}'/(n+1)$.
By C4, $P_n=Q_n+r_n^{[1]}Q_{n-1}+s_n^{[1]}Q_{n-2}$ if $n\geq2$.
Hence, since $\{Q_n\}_{n\geq0}$ is a monic OP with respect to ${\bf v}=\phi{\bf u}$,
we deduce, for each $n\geq2$,
\begin{align*}
r_n^{[1]}&=\frac{\langle{\bf u},\phi  P_n P_n^{\prime}\rangle }{\langle{\bf u},\phi  P_n^{\prime} P_{n-1}\rangle }\\
&=\frac{\langle{\bf u}, P_{n-1}^2\rangle }{\langle{\bf u},\phi  P_n^{\prime} P_{n-1}\rangle }
\frac{\langle{\bf u}, \phi  P_n^{\prime} P_n\rangle }{\langle{\bf u}, P_n^2\rangle }
\frac{\langle{\bf u}, P_n^2\rangle }{\langle{\bf u}, P_{n-1}^2\rangle }=\frac{1}{c_n} b_n \gamma_n
=\mbox{$\frac{1}{2}$}\frac{\psi(\beta_n)}{d_{n-1}}\,,
\end{align*}
where the third equality holds taking into account C2. 
Similarly, for each $n\geq2$,
\begin{align*}
s_n^{[1]}&=\frac{a\langle{\bf u}, P_n^2\rangle }{\frac{1}{n-1}\langle{\bf u},\phi  P_{n-1}^{\prime} P_{n-2}\rangle}
=\frac{(n-1)a \langle{\bf u}, P_n^2\rangle}{c_{n-1}\langle{\bf u}, P_{n-2}^2\rangle}\\
&=\frac{(n-1)a}{c_{n-1}}\gamma_{n-1}\gamma_n=-\frac{(n-1)a}{d_{n-2}}\gamma_n \;. 
\end{align*}

(C1$\,\Rightarrow\,$C5).
By hypothesis, $D(\phi{\bf u})=\psi{\bf u}$, where
$\phi\in\mathcal{P}_2$, $\psi\in\mathcal{P}_1$, and $\deg\psi=1$ (cf. Lemma \ref{poly-grau21}).
Fix $n\in\mathbb{N}$, and write
\begin{equation}\label{C1C5C3eq}
\phi P_n^{\prime\prime}+\psi P_n'=\sum_{j=0}^n\lambda_{n,j}P_j\;.  
\end{equation}
Then, for each $j$ such that $0\leq j\leq n$,
\begin{align*}
\langle{\bf u}, P_j^2\rangle\lambda_{n,j} &=
\langle {\bf u},(\phi P_n^{\prime\prime}+\psi P_n')P_j\rangle =
\langle \phi{\bf u},P_n^{\prime\prime}P_j\rangle +\langle\psi{\bf u}, P_n'P_j\rangle \\ 
&= \langle \phi{\bf u},(P_n'P_j)'\rangle-\langle \phi{\bf u},P_n'P_j'\rangle+\langle\psi{\bf u}, P_n'P_j\rangle
=-\langle \phi{\bf u},P_n'P_j'\rangle\,.
\end{align*}
Since by hypothesis C1 holds, and we have already proved that
C1$\,\Rightarrow\,$C1$'$$\,\Rightarrow\,$C2$\,\Rightarrow\,$C3, and in the proof of
C2$\,\Rightarrow\,$C3 we have shown that $\{Q_n=P'_{n+1}/(n+1)\}_{n\geq0}$
is a monic OP with respect to ${\bf v}=\phi{\bf u}$,
then $\langle \phi{\bf u},P_n'P_j'\rangle=0$ if $j\neq n$,
hence (\ref{C1C5C3eq}) reduces to
\begin{equation}\label{C1C5C3eqA}
\phi P_n^{\prime\prime}+\psi P_n'+\lambda_nP_n=0\;,\quad n\geq0\;,
\end{equation}
where $\lambda_n=-\lambda_{n,n}$.
Comparing leading coefficients in (\ref{C1C5C3eqA}), and setting
$\phi(x)=ax^2+bx+c$ and $\psi(x)=px+q$, we obtain
$\lambda_n=-n\big((n-1)a+p\big)=-nd_{n-1}$, hence $\lambda_n\neq0$ if $n\geq1$
(since C1$\,\Rightarrow\,$C1$'$, so $(\phi,\psi)$ is an admissible pair).
Thus C1$\,\Rightarrow\,$C5.
\smallskip

(C5$\,\Rightarrow\,$C1).
By hypothesis, there extist $\phi,\psi\in\mathcal{P}$, and $\lambda_n\in\mathbb{C}$, with $\lambda_n\neq0$ if $n\geq1$,
such that $-\phi P_{n+1}^{\prime\prime}=\psi P_{n+1}'+\lambda_{n+1} P_{n+1}$.
Taking in this equation $n=0$ and $n=1$ we deduce $\psi=-\lambda_1P_1\in\mathcal{P}_1\setminus\mathcal{P}_0$
and $\phi=-(\psi P_2'+\lambda_2P_2)/2\in\mathcal{P}_2$.
We will prove that $D(\phi{\bf u})=\psi{\bf u}$ by showing that the actions of the functionals
$D(\phi{\bf u})$ and $\psi{\bf u}$ coincide on the simple set $\{Q_n\}_{n\geq0}$. 
Indeed,
\begin{align*}
\langle D(\phi{\bf u}),Q_n\rangle &=\frac{1}{n+1} \langle D(\phi{\bf u}),P'_{n+1}\rangle\\
&= -\frac{1}{n+1}\langle{\bf u},\phi P_{n+1}^{\prime\prime}\rangle=\frac{1}{n+1}\langle{\bf u},\psi P'_{n+1}+\lambda_{n+1}P_{n+1}\rangle \\
& = \langle{\bf u}, \psi Q_n\rangle+\frac{\lambda_{n+1}}{n+1} \langle{\bf u},P_{n+1}\rangle=\langle\psi{\bf u}, Q_n\rangle.
\end{align*}
Since at least one of the polynomials $\phi$ and $\psi$ is nonzero (because $\lambda_n\neq0$), C1 holds.
\smallskip

(C2$\,\Rightarrow\,$C7).
Since by hypothesis ({\rm C2}) holds, we may write
\begin{eqnarray}
\phi P_n^{\prime} =a_nP_{n+1}+b_nP_n+c_nP_{n-1} \; , \quad \label{c.27} \\ [0.5em]
\phi P_{n-1}^{\prime} =a_{n-1}P_n+b_{n-1}P_{n-1}+c_{n-1}P_{n-2}\;.
\label{c.28}
\end{eqnarray}
Multiplying (\ref{c.27}) by $P_{n-1}$ and (\ref{c.28}) by $P_n$ and
adding the resulting equalities, we find that $\phi(P_n P_{n-1})^{\prime}$
is a linear combination of the polynomials
$P_n^2$, $P_nP_{n-1}$, $P_{n-1}^2$, $P_{n+1}P_{n-1}$ and $P_nP_{n-2}$.
Substituting $P_{n+1}$ and $P_{n-2}$ by the corresponding expressions given by the TTRR, we deduce
\begin{equation}
\phi(P_nP_{n-1})^{\prime}=A_nP_n^2+(B_nx+C_n)P_nP_{n-1}+D_nP_{n-1}^2 \;, \quad n\geq1\;,
\label{c.29}
\end{equation}
where
$$
\begin{array}{l}
A_n=a_{n-1}-\frac{c_{n-1}}{\gamma_{n-1}} \; , \quad B_n=a_n+ \frac{c_{n-1}}{\gamma_{n-1}} \; , \\ [0.5em]
C_n=-a_n \beta_n +b_n+b_{n-1}-\frac{c_{n-1}}{\gamma_{n-1}}\beta_{n-1} \; , \quad
D_n=c_n-a_n\gamma_n \; .
\end{array}
$$
Write $\phi(x)=ax^2+bx+c$ and $\psi(x)=px+q$.
We have already seen that C2$\,\Leftrightarrow\,$C1$'$, and while proving
C1$'$$\,\Rightarrow\,$C2 we have shown that the coefficients $a_n$, $b_n$, and $c_n$ appearing in
(\ref{c.27}) are given by $a_n=na$, $b_n=-\frac12\psi(\beta_n)$, and $c_n=-d_{n-1}\gamma_n$.
It follows that
\begin{equation}\label{AnBnCnDn1}
\begin{array}{rcl}
A_n=d_{2n-3} \; , \quad B_n=2a-p \; , \quad D_n=-d_{2n-1}\gamma_n \; , \\ [0.5em]
C_n = -\frac{1}{2} \big( d_{2n}\beta_n -d_{2n-4}\beta_{n-1} \big) -q = b-q \; ,
\end{array}
\end{equation}
where the last equality is easily derived using the expressions
for the $\beta-$parameters given in the statement of the theorem.
Therefore, $B_n x+C_n=\phi^{\prime}-\psi$ (independent of $n$).
Finally, substituting (\ref{AnBnCnDn1}) into (\ref{c.29}) yields the equation
appearing in C7, being $h_n=A_n=d_{2n-3}$ and $t_n=D_n=-d_{2n-1}\gamma_n$ for each $n\geq1$.
Thus C2$\,\Rightarrow\,$C7.
\smallskip

(C7$\,\Rightarrow\,$C2).
Fix an integer $n\geq1$. 
For this $n$, rewrite the equation in {\rm (C7)} as
$$
\big(\phi P_n^{\prime}+\psi P_n-t_nP_{n-1}\big)P_{n-1}=
\big(-\phi P_{n-1}^{\prime}+\phi' P_{n-1}+h_nP_n\big)P_{n}\;.
$$
Therefore, since $P_n$ and $P_{n-1}$ have no common zeros,
there is 
$\pi_{1,n}\in\mathcal{P}_1$ such that
\begin{eqnarray}
\phi P_n^{\prime}+\psi P_n-t_nP_{n-1}=\pi_{1,n} P_n \; , \label{Eq2} \\ 
-\phi P_{n-1}^{\prime}+\phi' P_{n-1}+h_nP_n=\pi_{1,n}P_{n-1}\;.\label{Eq3}
\end{eqnarray}
By comparing the leading coefficients on both sides of equation (\ref{Eq2}) we deduce
$\pi_{1,n}(x)=d_nx+z_n$ for some $z_n\in\mathbb{C}$ (and $d_n=na+p$).
By hypothesis, $(\phi,\psi)$ is an admissible pair, hence $d_n\neq0$ 
and so $\deg\pi_{1,n}=1$. 
Moreover, by the TTRR for $\{P_n\}_{n\geq0}$, $xP_n=P_{n+1}+\beta_nP_n+\gamma_{n}P_{n-1}$.
Therefore, (\ref{Eq2}) may be rewritten as
$$
\phi P_n^{\prime} =a_nP_{n+1}+b_nP_n+c_nP_{n-1} \; , \label{Eq4}
$$
where $a_n=na$, $b_n=na\beta_n+z_n-q$, and $c_n=na\gamma_n+t_n$.
To conclude the proof we need to show that $c_n\neq0$ for all $n\geq1$.
Indeed, changing $n$ into $n+1$ in (\ref{Eq3}) and adding the resulting equation with (\ref{Eq2}), we obtain
$$
(\psi+\phi') P_n-t_nP_{n-1}+h_{n+1}P_{n+1} =
\big((d_n+d_{n+1})x+(z_n+z_{n+1})\big)P_n \; . \label{Eq5}
$$
Since $\psi+\phi'=(2a+p)x+q+b$ and taking into account once again the TTRR for $\{P_n\}_{n\geq0}$,
the last equation may be rewritten as a trivial linear combination
of the three polynomials $P_{n+1}$, $P_{n}$, and $P_{n-1}$.
Thus, we deduce
$$
h_{n+1}=d_{2n-1}\;,\quad
z_{n+1}=-z_n-d_{2n-1}\beta_n+q+b\;,\quad
t_n=-d_{2n-1}\gamma_n\;.
$$
Therefore, $c_n=na\gamma_n+t_n=-d_{n-1}\gamma_n\neq0$ 
(since $n\geq1$).
This completes the proof.
\end{proof}

\begin{remark}\em
The parameters $\beta_n$ and $\gamma_n$ appearing in Theorem \ref{ThmClassicalOP}
may be written explicitly in terms of the coefficients of $\phi$ and $\psi$ as follows:
$$
\begin{array}{c}
\beta_n=\displaystyle-\frac{(-2a+p)q+2bn[(n-1)a+p]}{(2na+p)[(2n-2)a+p]}\;, \\ [1.25em]
\gamma_{n+1}=\displaystyle\frac{-(n+1)[(n-1)a+p][a(nb+q)^2-b(nb+q)(2na+p)+c(2na+p)^2]}{[(2n-1)a+p](2na+p)^2[2(n+1)a+p]}\;.
\end{array}
$$
\end{remark}

\subsection{Classification and canonical representatives}

We all always hear say:
{\it up to constant factors and affine changes of variables,
there are only four (parametric) families of classical OP,
namely, Hermite, Laguerre, Jacobi, and Bessel polynomials}. But, what is the rigorous meaning of this statement?
The corresponding regular functionals will be denoted by
${\bf u}_H$, ${\bf u}_L^{(\alpha)}$,
${\bf u}_J^{(\alpha,\beta)}$, and ${\bf u}_B^{(\alpha)}$ (resp.) and these will be called
the {\it canonical representatives} (or {\it canonical forms}) of the classical functionals.
Their description is given in Table \ref{Table1}.
Each one of these functionals fulfils $(\ref{EDClassic1})$,
being the corresponding pair $(\phi,\psi)\equiv(\Phi,\Psi)$ given in the table.
The regularity conditions in the table are determined by conditions (ii) appearing in
$(\ref{P-regular1})$.

\begin{table}
\centering
\begin{tabular}{|>{\columncolor[gray]{0.95}}c|c|c|c|c|}
\hline \rowcolor[gray]{0.95}
\rule{0pt}{1.2em} Class & ${\bf u}$ & $\Phi$ & $\Psi$ & regularity conditions \\
\hline
\rule{0pt}{1.2em} Hermite & ${\bf u}_H$ & $1$& $-2x$ & ------ \\
\rule{0pt}{1.2em} Laguerre & ${\bf u}_L^{(\alpha)}$ & $x$ & $-x+\alpha +1$ &
$-\alpha\not\in\mathbb{N}$ \\
\rule{0pt}{1.2em} Jacobi & ${\bf u}_J^{(\alpha,\beta)}$ & $1-x^2$ & $-(\alpha + \beta +2)x +\beta-\alpha $ &
$-\alpha,-\beta,-(\alpha+\beta+1)\not\in\mathbb{N}$ \\
\rule{0pt}{1.2em} Bessel & ${\bf u}_B^{(\alpha)}$ & $x^2$ & $(\alpha +2)x +2$  &
$-(\alpha+1)\not\in\mathbb{N}$ \\
\hline
\end{tabular}
\medskip
\caption{Classification and canonical forms of the classical functionals}\label{Table1}
\end{table}

Ultimately, denoting by $[{\bf u}]$ the equivalent class
determined by a functional ${\bf u}\in\mathcal{P}'$ (see \cite[Section 3.1.2.4, pp. 18-19]{Maroni1994}), and setting
$\mathcal{P}'_C=\{{\bf u}\in\mathcal{P}'\,|\,\mbox{\rm ${\bf u}$ is classical}\}$,
we will show that
$$
\mathcal{P}'_C/_\sim=\big\{\,[{\bf u}]\,|\,{\bf u}\in\mathcal{P}'_C\big\}=
\big\{[{\bf u}_H], [{\bf u}_L^{(\alpha)}],
[{\bf u}_J^{(\alpha,\beta)}], [{\bf u}_B^{(\alpha)}]\,\big\}\;,
$$
where the parameters $\alpha$ and $\beta$ vary on $\mathbb{C}$ subject to the regularity conditions
in Table \ref{Table1}, and $\sim$ is an equivalence relation in $\mathcal{P}'$ defined by
\begin{equation}\label{u-equiv-v1}
{\bf u}\sim{\bf v}\qquad\mbox{\rm iff}\qquad
\exists A\in\mathbb{C}\setminus\{0\}\;,\;\; \exists B\in\mathbb{C}\;:\;\;
{\bf v}=\big(\bm{h}_{A^{-1}}\circ\bm{\tau}_{-B}\big){\bf u}\;.
\end{equation}
 
We start by proving a proposition that allow us to ensure that this equivalence relation
preserves the classical character of a given classical functional.

\begin{lemma}\label{equiv-u-v}
Let ${\bf u},{\bf v}\in\mathcal{P}'$ and suppose that
${\bf u}\sim{\bf v}$, i.e., $(\ref{u-equiv-v1})$ holds.
Suppose that there exist two polynomials $\phi$ and $\psi$ such that
$$D(\phi{\bf u})=\psi{\bf u}\;.$$
Let $\Phi(x)=K\phi(Ax+B)$ and $\Psi(x)=KA\psi(Ax+B)$, being $K\in\mathbb{C}\setminus\{0\}$.
Then
$$D(\Phi{\bf v})=\Psi{\bf v}\;.$$
Moreover, if ${\bf u}$ is a classical functional, then so is ${\bf v}$.
\end{lemma}

\begin{proof}
Since ${\bf u}$ and ${\bf v}$ fulfill (\ref{u-equiv-v1}), then
$$
\langle {\bf v},x^n\rangle=\Big\langle {\bf u},\Big(\mbox{$\frac{x-B}{A}$}\Big)^n\,\Big\rangle\;,
\quad n\in\mathbb{N}_0\;.
$$
Therefore, for each $n\in\mathbb{N}_0$, we have
\begin{align*}
\langle D(\Phi{\bf v}),x^n\rangle &= -n\langle {\bf v},\Phi(x)x^{n-1}\rangle
=-n\big\langle {\bf u},\big(\tau_{B}\circ h_{A^{-1}}\big)\big(\Phi(x)x^{n-1}\big)\big\rangle \\ 
 &= -n\Big\langle {\bf u},\Phi\Big(\mbox{$\frac{x-B}{A}$}\Big)\Big(\mbox{$\frac{x-B}{A}$}\Big)^{n-1}\Big\rangle\\
&=-\Big\langle{\bf u},K\phi(x)\cdot A\frac{{\rm d}}{{\rm d}x}\Big\{\Big(\mbox{$\frac{x-B}{A}$}\Big)^{n}\Big\}\Big\rangle\\ 
 &= KA\Big\langle D\big(\phi(x){\bf u}\big),\Big(\mbox{$\frac{x-B}{A}$}\Big)^{n}\Big\rangle
 =KA\Big\langle \psi(x) {\bf u},\Big(\mbox{$\frac{x-B}{A}$}\Big)^{n}\Big\rangle \\ 
 &= \Big\langle {\bf u},\Psi\Big(\mbox{$\frac{x-B}{A}$}\Big)\Big(\mbox{$\frac{x-B}{A}$}\Big)^{n}\Big\rangle
 =\big\langle {\bf u},\big(\tau_{B}\circ h_{A^{-1}}\big)\big(\Psi(x)x^{n}\big)\big\rangle \\ 
 &= \langle {\bf v},\Psi(x) x^n\rangle=\langle \Psi{\bf v},x^n\rangle \;.
\end{align*}

Finally, the last sentence stated in the lemma follows by using Theorem \ref{OPequiv}.
\end{proof}

\begin{theorem}[canonical representatives of the classical functionals]\label{canonic-forms-classical}
Let ${\bf u}\in\mathcal{P}^\prime$ be a classical functional,
so that ${\bf u}$ fulfils
\begin{equation}\label{Pears1}
D(\phi{\bf u})=\psi{\bf u}\;,
\end{equation}
where $\phi(x)=ax^2+bx+c$ and $\psi(x)=px+q$, subject to the regularity conditions
\begin{equation}\label{Pears1reg}
na+p\neq0\;,\quad\phi\left(-\frac{nb+q}{2na+p}\right)\neq0\;,
\quad\forall n\in\mathbb{N}_0\;.
\end{equation}
Then, there exists a regular functional ${\bf v}\in\mathcal{P}'$ such that
\begin{equation}\label{Pears2}
{\bf u}\sim{\bf v}\;,\quad D(\Phi{\bf v})=\Psi{\bf v}\;,
\end{equation}
where, for each classical functional determined by the pair $(\phi,\psi)$,
the corresponding pair $(\Phi,\Psi)$ is given by Table \ref{Table1}.
More precisely, setting
$$\Delta=b^2-4ac\;;\quad d=\psi\left(-\mbox{$\frac{b}{2a}$}\right)\;\;\mbox{\rm if}\;\; a\neq0 \;,$$
the following holds:
\begin{itemize}
\item[\,]\hspace*{-2em}{\rm 1.\,(Hermite)} if $a=b=0$, then:
$$
{\bf v}=\big(\bm{h}_{\sqrt{-p/(2c)}}\circ\bm{\tau}_{q/p}\big){\bf u}={\bf u}_{{}_H}\;;
$$
\item[\,]\hspace*{-2em}{\rm 2.\,(Laguerre)} if $a=0$ and $b\neq0$, then:
$$
{\bf v}=\big(\bm{h}_{-p/b}\circ\bm{\tau}_{c/b}\big){\bf u}={\bf u}_{{}_L}^{(\alpha)}\;,\quad\alpha=-1+(qb-pc)/b^2\;;
$$
\item[\,]\hspace*{-2em}{\rm 3.\,(Bessel)} if $a\neq0$ and $\Delta=0$, then:
$$
{\bf v}=\big(\bm{h}_{2a/d}\circ\bm{\tau}_{b/(2a)}\big){\bf u}={\bf u}_{{}_B}^{(\alpha)}\;,\quad\alpha=-2+p/a\;;
$$
\item[\,]\hspace*{-2em}{\rm 4.\,(Jacobi)} if $a\neq0$ and $\Delta\neq0$, then:
$$
{\bf v}=\big(\bm{h}_{-2a/\sqrt{\Delta}}\circ\bm{\tau}_{b/(2a)}\big){\bf u}={\bf u}_{{}_J}^{(\alpha,\beta)}\;,
$$
$$
\alpha=-1+p/(2a)-d/\sqrt{\Delta}\;,\quad
\beta=-1+p/(2a)+d/\sqrt{\Delta}\;.
$$
\end{itemize}
\end{theorem}

\begin{proof}
Taking into account Lemma \ref{equiv-u-v},
the theorem will be proved if we are able to show that,
for each given pair $(\phi,\psi)$, and for each
corresponding pair $(\Phi,\Psi)$ given by Table \ref{Table1}---where the
``corresponding pair'' $(\Phi,\Psi)$ is the one in the table such that $\phi$ and $\Phi$
have the same degree and their zeros the same multiplicity---,
there exist $A,K\in\mathbb{C}\setminus\{0\}$ and $B\in\mathbb{C}$ such that the relations
\begin{equation}\label{PhiPsiK1}
\Phi(x)=K\phi(Ax+B)\;,\quad \Psi(x)=KA\psi(Ax+B)=KA^2px+KA(Bp+q)
\end{equation}
hold, for appropriate choices of the parameters $\alpha$ and $\beta$
appearing in Table \ref{Table1} for the Laguerre, Bessel, and Jacobi cases.
Indeed, considering the four possible cases 
determined by the polynomial $\phi$, we have:
\smallskip

1. Assume $a=b=0$, i.e., $\phi(x)=c$. The regularity conditions (\ref{Pears1reg})
ensure that $p\neq0$ and $c\neq0$. Therefore, since in this case we require
$(\Phi,\Psi)=(1,-2x)$, from (\ref{PhiPsiK1}) we obtain the equations
$$
1=Kc\;,\quad -2=KA^2p\;,\quad 0=Bp+q\;.
$$
A solution of this system of equations is
$$
K=1/c\;,\quad A=\sqrt{-2c/p}\;,\quad B=-q/p\;,
$$
which gives the desired result for the Hermite case, by Lemma \ref{equiv-u-v}.
\smallskip

2. Assume $a=0$ and $b\neq0$, so that $\phi(x)=bx+c$.
Since in this case we require $(\Phi,\Psi)=(x,-x+\alpha+1)$, from (\ref{PhiPsiK1}) we obtain
$$
1=KAb\;,\quad 0=bB+c\;,\quad -1=KA^2p\;,\quad \alpha+1=KA(Bp+q)\;.
$$
Solving this system we find
$$
K=-p/b^2\;,\quad B=-c/b\;,\quad A=-b/p\;,\quad \alpha=-1+(qb-pc)/b^2\;.
$$
Notice that, in this case,
$$
d_n=p\;,\quad \phi\left(-\frac{nb+q}{2na+p}\right)
=-\frac{b^2}{p}\big(n+\alpha+1\big)\;,
$$
hence the regularity conditions (\ref{Pears1reg}) ensure that $p\neq0$
(and so $K$ and $A$ are well defined, being both nonzero complex numbers) and $-\alpha\not\in\mathbb{N}$.
\smallskip

3. Assume $a\neq0$ and $\Delta=0$. Then $\phi(x)=a\Big(x+\frac{b}{2a}\Big)^2$.
In this case we require $(\Phi,\Psi)=\big(x^2,(\alpha+2)x+2\big)$, hence from (\ref{PhiPsiK1}) we obtain
$$
1=KA^2a\;,\quad 0=B+b/(2a)\;,\quad \alpha+2=KA^2p\;,\quad 2=KA(Bp+q)\;.
$$
Therefore, taking into account that
$d=\psi\left(-\mbox{$\frac{b}{2a}$}\right)=(2aq-pb)/(2a)$, we deduce
$$
K=4a/d^2\;,\quad B=-b/(2a)\;,\quad A=d/(2a)\;,\quad \alpha=-2+p/a\;.
$$
In this case we have
$$
d_n=a(n+\alpha+2)\;,\quad
\phi\left(-\frac{nb+q}{2na+p}\right)=\frac{d^2}{a(2n+\alpha+2)^2}\;,
$$
hence conditions (\ref{Pears1reg}) ensure that $-(\alpha+1)\not\in\mathbb{N}$
and $d\neq0$, and so, in particular, $K$ is well defined, being both $K$ and $A$ nonzero complex numbers.
\smallskip

4. Finally, assume $a\neq0$ and $\Delta\neq0$.
Writing $\phi(x)=a\Big[\Big(x+\frac{b}{2a}\Big)^2-\frac{\Delta}{4a^2}\Big]$, since
in this case we require $(\Phi,\Psi)=\big(1-x^2,-(\alpha+\beta+2)x+\beta-\alpha\big)$,
from (\ref{PhiPsiK1}) we obtain
$$
\begin{array}{c}
-1=KA^2a\;,\quad 0=B+b/(2a)\;,\quad 1=Ka\Big[\Big(B+\frac{b}{2a}\Big)^2-\frac{\Delta}{4a^2}\Big]\;, \\ [0.5em]
-(\alpha+\beta+2)=KA^2p\;,\quad \beta-\alpha=KA(Bp+q)\;.
\end{array}
$$
A solution of this system of five equations is
$$
\begin{array}{c}
K=-4a/\Delta\;,\quad B=-b/(2a)\;,\quad A=-\sqrt{\Delta}/(2a)\;, \\ [0.25em]
\alpha=-1+p/(2a)-d/\sqrt{\Delta}\;,\quad \beta=-1+p/(2a)+d/\sqrt{\Delta}\;.
\end{array}
$$
(We choose $A$ with the minus sign since whenever $(\phi,\psi)=(\Phi,\Psi)$ that choice implies $A=1$
and $B=0$, hence ${\bf u}={\bf v}={\bf u}_J^{(\alpha,\beta)}$, and so it is a more natural choice.)
Adding and subtracting the last two equations for $\alpha$ and $\beta$,
we find $\alpha+\beta+2=p/a$ and $\alpha-\beta=-2d/\sqrt{\Delta}$, hence we deduce
$$
d_n=a(n+\alpha+\beta+2)\;,\quad
\phi\left(-\frac{nb+q}{2na+p}\right)=-\frac{\Delta}{a}\frac{(n+\alpha+1)(n+\beta+1)}{(2n+\alpha+\beta+2)^2}\;,
$$
Therefore, conditions (\ref{Pears1reg}) ensure that $-(\alpha+\beta+1)\not\in\mathbb{N}$,
$-\alpha\not\in\mathbb{N}$, and $-\beta\not\in\mathbb{N}$. This completes the proof.
\end{proof}

\begin{remark}\em
It follows from the proof of Theorem \ref{canonic-forms-classical} that
the parameters $\alpha$ and $\beta$ defined in the statement of this theorem (in cases 2, 3, and 4)
fulfil the regularity conditions appearing in Table \ref{Table1}.
\end{remark}

The preceding theorem allows us to classify each classical functional according with
the degree of the polynomial $\phi$ appearing in equation $(\ref{EDClassic1})$.

\begin{corollary}
Let ${\bf u}$ be a classical functional, fulfilling $(\ref{EDClassic1})$--$(\ref{grauPhiPsi})$.
\begin{enumerate}
\item[{\rm (i)}] if $\deg\phi=0$ (hence $\phi$ is a nonzero constant), then ${\bf u}\sim{\bf u}_H\,$;
\item[{\rm (ii)}] if $\deg\phi=1$, then ${\bf u}\sim{\bf u}_L^{(\alpha)}\,$ for some $\alpha$;
\item[{\rm (iii)}] if $\deg\phi=2$ and $\phi$ has simple zeros, then ${\bf u}\sim{\bf u}_J^{(\alpha,\beta)}\,$ for some pair $(\alpha,\beta)$;
\item[{\rm (iv)}] if $\deg\phi=2$ and $\phi$ has a double zero, then ${\bf u}\sim{\bf u}_B^{(\alpha)}\,$ for some $\alpha$.
\end{enumerate}
\end{corollary}

The monic OP with respect to the canonical representatives ${\bf u}_H$, ${\bf u}_L^{(\alpha)}$,
${\bf u}_J^{(\alpha,\beta)}$, and ${\bf u}_B^{(\alpha)}$ will be denoted by
$\{\widehat{H}_n\}$, $\{\widehat{L}_n^{(\alpha)}\}$, $\{\widehat{P}_n^{(\alpha,\beta)}\}$,
and $\{\widehat{B}_n^{(\alpha)}\}$ (resp.), and they will be called the (monic)
{\sl Hermite}, {\sl Laguerre}, {\sl Jacobi}, and {\sl Bessel} polynomials.
Table \ref{Table3} summarizes the corresponding parameters appearing in all
characterizations presented in Theorem \ref{ThmClassicalOP}. In view of Theorem \ref{canonic-forms-classical} and Theorem \ref{OPequiv},
we may now justify a sentence made at the beginning of the section regarding Hermite, Laguerre, Jacobi, and Bessel polynomials. As Maroni said in an interview when asked about Bessel polynomials: {\em ``comme dans le roman d'Alexandre Dumas, les trois mousquetaires \'etaient quatre en r\'ealité''}.

\begin{table}
\centering
\begin{tabular}{|>{\columncolor[gray]{0.95}}c|c|c|c|c|}
\hline \rowcolor[gray]{0.95}
\rule{0pt}{1.2em} & $\widehat{H}_n$ & $\widehat{L}_n^{(\alpha)}$ & $\widehat{P}_n^{(\alpha,\beta)}$ & $\widehat{B}_n^{(\alpha)}$\\
\hline
\rule{0pt}{1.2em} $\lambda_n$ & ${\scriptstyle{2n}}$ & ${\scriptstyle{n}}$ & ${\scriptstyle{n(n+\alpha+\beta+1)}}$ & ${\scriptstyle{-n(n+\alpha+1)}}$ \\
\rule{0pt}{1.5em} $\beta_n$ & ${\scriptstyle{0}}$ & ${\scriptstyle{2n+\alpha+1}}$ &
    $\frac{\beta^2-\alpha^2}{(2n+\alpha+\beta)(2n+2+\alpha+\beta)}$ &
    $\frac{-2\alpha}{(2n+\alpha)(2n+2+\alpha)}$  \\
\rule{0pt}{1.75em} $\gamma_n$ & $\frac{n}{2}$ & ${\scriptstyle{n(n+\alpha)}}$ &
    $\frac{4n(n+\alpha)(n+\beta)(n+\alpha+\beta)}{(2n+\alpha+\beta-1)
    (2n+\alpha+\beta)^2(2n+\alpha+\beta+1)}$ &
    $\frac{-4n(n+\alpha)}{(2n+\alpha-1)(2n+\alpha)^2(2n+\alpha+1)}$ \\
\rule{0pt}{1em} $a_n$ & ${\scriptstyle{0}}$ & ${\scriptstyle{0}}$ & ${\scriptstyle{-n}}$ & ${\scriptstyle{n}}$ \\
\rule{0pt}{1.25em} $b_n$ & ${\scriptstyle{0}}$ & ${\scriptstyle{n}}$ & $\frac{2(\alpha-\beta)n(n+\alpha+\beta+1)}{(2n+\alpha+\beta)(2n+2+\alpha+\beta)}$ & $\frac{-4n(n+\alpha+1)}{(2n+\alpha)(2n+2+\alpha)}$ \\
\rule{0pt}{1.75em} $c_n$ & ${\scriptstyle{n}}$ & ${\scriptstyle{n(n+\alpha)}}$ & $\frac{4n(n+\alpha)(n+\beta)(n+\alpha+\beta)(n+\alpha+\beta+1)}{(2n+\alpha+\beta-1)
    (2n+\alpha+\beta)^2(2n+\alpha+\beta+1)}$ & $\frac{4n(n+\alpha)(n+\alpha+1)}{(2n+\alpha-1)(2n+\alpha)^2(2n+\alpha+1)}$ \\
\rule{0pt}{1.5em} $r_n^{[1]}$ & ${\scriptstyle{0}}$ & ${\scriptstyle{n}}$ & $\frac{2(\alpha-\beta)n}{(2n+\alpha+\beta)(2n+2+\alpha+\beta)}$ & $\frac{4n}{(2n+\alpha)(2n+2+\alpha)} $ \\
\rule{0pt}{1.75em} $s_n^{[1]}$ & ${\scriptstyle{0}}$ & ${\scriptstyle{0}}$ & $\frac{-4(n-1)n(n+\alpha)(n+\beta)}{(2n+\alpha+\beta-1)(2n+\alpha+\beta)^2(2n+\alpha+\beta+1)}$ & $\frac{4(n-1)n}{(2n+\alpha-1)(2n+\alpha)^2(2n+\alpha+1)}$ \\
\rule{0pt}{1.2em} $h_n$ & ${\scriptstyle{-2}}$ & ${\scriptstyle{-1}}$ & ${\scriptstyle{-(2n+\alpha+\beta-1)}} $ & ${\scriptstyle{2n+\alpha-1}} $ \\
\rule{0pt}{1.5em} $t_n$ & ${\scriptstyle{n}}$ & ${\scriptstyle{n(n+\alpha)}}$ & $\frac{4n(n+\alpha)(n+\beta)(n+\alpha+\beta)}{(2n+\alpha+\beta-1)(2n+\alpha+\beta)^2}$ & $\frac{4n(n+\alpha)}{(2n+\alpha-1)(2n+\alpha)^2}$ \\
\rule{0pt}{1.5em} $k_n$ & $\frac{(-1)^n}{2^n}$ & ${\scriptstyle{(-1)^n}}$ & $\frac{(-1)^n}{(n+\alpha+\beta+1)_n}$ & $\frac{1}{(n+\alpha+1)_n}$  \\ [0.5em]
\hline
\end{tabular}
\medskip
\caption{Parameters for the classical monic OP
appearing in Theorem \ref{ThmClassicalOP} with respect to the canonical forms 
given in Table \ref{Table1}.}\label{Table3}
\end{table}


\section*{Acknowledgements}
This work is supported by the Centre for Mathematics of the University of Coimbra-UIDB/00324/2020, funded by the Portuguese Government through FCT/ MCTES.

\end{document}